\documentclass[journal, 10 pt]{article}  

\usepackage{amsfonts}

\usepackage{graphicx}
\usepackage{caption}
\usepackage{epstopdf}
\usepackage{mathtools}
\usepackage{amsthm}
\usepackage{amssymb, amsopn}
\usepackage{mathrsfs}
\usepackage{arXivmacros} 
\usepackage{url}
\usepackage{mdwmath}
\usepackage{mdwtab}

\usepackage{ stmaryrd }

\usepackage[all]{xy}

\usepackage[normalem]{ulem}

\usepackage{enumitem}
\setlist[enumerate]{leftmargin=.5in}
\setlist[itemize]{leftmargin=.5in}

\newcommand{\insertTitle}{Global Stability of a Class of Difference Equations on Solvable Lie Algebras}

\title{\insertTitle
}

\date{\vspace{-5ex}}

\author{Philip James McCarthy\thanks{Partially funded by the Ontario Graduate Scholarship (OGS) and the Natural Sciences and Engineering Research Council of Canada (NSERC).} \qquad Christopher Nielsen
  \thanks{Supported by the Natural Sciences
    and Engineering Research Council of Canada (NSERC). } \thanks{The
    authors are with the Dept. of Electrical and Computer
    Engineering, University of Waterloo, Waterloo ON, N2L 3G1 Canada. 
    }
}

\begin{document}

\maketitle

\begin{abstract}
Motivated by the ubiquitous sampled-data setup in applied control, we examine the stability of a class of difference equations that arises by sampling a right- or left-invariant flow on a matrix Lie group. The map defining such a difference equation has three key properties that facilitate our analysis: 1) its power series expansion enjoys a type of strong convergence; 2) the origin is an equilibrium; 3) the algebraic ideals enumerated in the lower central series of the Lie algebra are dynamically invariant. We show that certain global stability properties are implied by stability of the Jacobian linearization of dynamics at the origin. In particular global asymptotic stability. If the Lie algebra is nilpotent, then the origin enjoys semiglobal exponential stability.
\end{abstract}



\section{Introduction}

We examine the stability of a class of difference equations that arises by sampling a right- or left-invariant flow on a matrix Lie group. There are many dynamical systems whose state spaces are naturally modelled as matrix Lie groups. Networks of oscillators can be modelled on~$\SO{2}^n$~\cite{Dorfler2014}. The group $\SE{3}$ captures the kinematics of rigid bodies in space, such as underwater vehicles~\cite{Leonard1997}, UAVs~\cite{Lee2010}, and robotic arms~\cite{Spong2006}. Robots exhibiting planar motion can be modelled on $\SE{2}$~\cite{Jin2014}. The unitary groups $\mathsf{U}(n)$ and $\mathsf{SU}(n)$~\cite{Petersen2010} can be used to model the evolution of quantum systems. Even the noise responses of some circuits evolve on Lie groups~\cite{Willsky1976}, specifically the solvable Lie group of invertible upper-triangular matrices.

Our main results -- Theorems~\ref{thm:Nilpotent},~\ref{thm:Solvable},~\ref{thm:Deadbeat}, and Corollaries~\ref{cor:SolvGAS} and~\ref{cor:SolvSGES} -- assert that there exists a sufficiently small spectral radius of the Jacobian linearization of the dynamics that implies various global stability properties of the origin, the weakest of which is global asymptotic stability. Lyapunov's Second Method can be used to establish local stability of an equilibrium, and it is a strong and surprising result when this method establishes global stability for a class of dynamical systems. In the continuous-time case, the Markus-Yamabe Conjecture~\cite{Meisters1996} supposes that global attractivity of a (unique) equilibrium is implied by the Jacobian of the vector field being everywhere Hurwitz; this conjecture is true for vector fields on $\Real^2$, but is in general false. The discrete-time analog of the Conjecture -- the key difference being that one supposes that the Jacobian is everywhere Schur -- similarly to the continuous-time case, is true for polynomial maps on $\Real^2$~\cite[Theorem B]{Cima1999} and in general false on $\Real^n$, $n \geq 3$. However, it is true for triangular maps on $\Real^n$~\cite[Theorem A]{Cima1999}. Again in continuous-time, Krasovskii's Method~\cite[p. 183]{Khalil2002} asserts that if there exists a symmetric positive definite $P \in \Real^{n \times n}$ that solves the Lyapunov equation for the Jacobian linearization at all $x \in \Real^n$, then the (unique) equilibrium is globally asymptotically stable.

In this paper we study dynamics on solvable Lie algebras. A Lie algebra is \emph{solvable} if and only if its derived length (see Definition~\ref{def:Solvable}) is finite. The complementary classification of Lie algebras is called \emph{semi-simple}, which is defined as those Lie algebras whose maximal solvable ideal -- the \emph{radical} -- is zero. Any Lie algebra $\g$ admits a \emph{Levi decomposition}, $\g = \mf l \inplus \mf r$, where $\mf r$ is the radical of $\g$, $\mf l$ is a semi-simple subalgebra of $\g$, and $\inplus$ means \emph{semidirect sum}\footnote{A detailed treatment of this decomposition can be found in, for example, \cite[\S$4$]{Onishchik1994} or~\cite[\S$3.14$]{Varadarajan1984}.}. This establishes that solvable Lie algebras are of fundamental importance in Lie theory. In control theory, it is possible to use solvable Lie algebras to approximate certain classes of vector fields~\cite{Crouch1984}. Nilpotent Lie algebras, a special case of solvable Lie algebras, can also be used for this purpose~\cite{Hermes1986,Struemper1998}.

In this paper, we study discrete-time dynamical systems of the form
\begin{equation}\label{eq:sys}
	X^+ = f(X,W),
\end{equation}
where $X \in \mc X := \g^n$, $n \geq 1$, $W \in \mc W := \g^r$, $r \geq 1$, and $f : \mc X \times \mc W \to \mc X$ is a \emph{Lie function} that belongs to \emph{class}-$\mc A$, which we define in Section~\ref{sec:f}. We make no general assumptions on the exogenous signal $W$, other than that it does not depend on $X$. We show that for this class of functions on solvable Lie algebras, global stability properties can be determined from the linear part of the dynamics.

\subsection{Step-Invariant Transforms}

In this section, we motivate the study of the class of systems described by~\eqref{eq:sys}, by showing that it arises naturally in the study of sampled-data control systems on Lie groups. In applied control, virtually all controllers are implemented using computers, and therefore evolve in discrete-time. The plant is often physical in nature and evolves in continuous-time. The combination of a discrete-time controller and a continuous-time plant is called \emph{sampled-data}. Figure~\ref{fig:SD} illustrates an example of this setup where the plant's dynamics evolve according to a right-invariant vector field on a matrix Lie group $\G$ with Lie algebra $\g$.

\begin{figure}[htb]
\centering
\includegraphics[width=0.75\columnwidth]{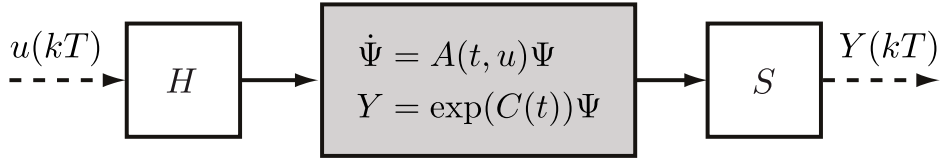}
\caption{Sampled-data right-invariant control system on a matrix Lie group $\G$.}
\label{fig:SD}
\end{figure}
%

The plant has state $\Psi(t) \in \G$, which evolves according to $\dot \Psi = A(t,u)\Psi$, where $A : \Real \times \Real^m \to \g$, and measured output $Y(t) \in \G$, which is defined by $Y = \exp(C(t))\Psi$, where $C : \Real \to \g$. The $H$ and $S$ blocks represent ideal zero-order hold and sample operations, respectively. By zero-order hold, we mean $(Hu)(t) = u(kT)$ for all $t \in [kT,(k + 1)T)$, where $k \in \Int$ is the discrete-time index and $T > 0$ is the sampling period. We assume that the hold and sample operations are synchronized and that $T$ is constant. The output $Y$ passes through the ideal sample, which is the identity map for all $t \in T\Int$, but is undefined otherwise. This sampled output is available for use by the controller, which generates the control signal $u[k] := u(kT)$. The discrete-time control signal passes through the ideal hold, yielding $u(t)$, which is piecewise constant. This piecewise constant control signal drives the plant.

The solution $\Psi(t)$ for the system in Figure~\ref{fig:SD} is given by the Magnus expansion~\cite{Blanes2009}, which provides an expression for $\Log(\Psi(t)) \in \g$ wherever the principal logarithm $\Log : \G \to \g$ is well-defined. Recall the adjoint operator of $A \in \g$, $\ad_A : \g \to \g$, $X \mapsto [A,X]$, where the Lie bracket $[\cdot,\cdot]$ is the commutator $AX - XA$. Define
%
%
\begin{equation*}
\begin{aligned}
	\Omega_1(t) &:= \int_0^tA(\tau,u(\tau))\mathrm d\tau \\
	\Omega_n(t) &:= \sum_{j = 1}^{n - 1}\frac{B_j}{j!}\sum_{\substack{k_1 + \cdots + k_j = n - 1 \\ k_1,\ldots,k_j \geq 1}}\int_0^t
		\ad_{\Omega_{k_1}(s)}\cdots\ad_{\Omega_{k_j}(s)}A(s,u(s))\mathrm ds, \quad n \geq 2,
\end{aligned}
\end{equation*}
where the $B_j$ are the Bernoulli numbers\footnote{Using the convention $B_1 = \frac{1}{2}$.}. Then 
\begin{equation}\label{eq:LocalMagnus}
	\Log(\Psi(t)) = \Log(\Psi(0)) + \sum_{n = 1}^\infty\Omega_n(t),
\end{equation}
which is a linear combination of the integral of $A$ and nested Lie brackets $\Omega_n(t)$, $n \geq 2$. In the sampled-data setup, due to the hold operator $H$, the plant is driven by a piecewise constant input signal. This motivates the \emph{step-invariant transform}, which is easily derived from~\eqref{eq:LocalMagnus}:
\begin{multline*}
	\Log(\Psi[k + 1]) = \Log(\Psi[k]) + \int_{kT}^{(k + 1)T}A(\tau,u[k])\mathrm d\tau \\
		+ \sum_{n = 2}^\infty\sum_{j = 1}^{n - 1}\frac{B_j}{j!}\sum_{\substack{k_1 + \cdots + k_j = n - 1 \\ k_1,\ldots,k_j \geq 1}}\int_{kT}^{(k + 1)T}
		\ad_{\Omega_{k_1}(s)}\cdots\ad_{\Omega_{k_j}(s)}A(s,u[k])\mathrm ds.
\end{multline*}

The solution simplifies significantly if for all $t_1,t_2 \in [kT,(k + 1)T)$, $A(t_1,u(t_1))$ commutes with $A(t_2,u(t_2))$\footnote{This is the case, for example, with the driftless kinematics of a rigid body with velocity inputs: ${A(t,u) = \sum_{i = 1}^mB_iu_i}$, $B_i \in \g$.}:
\begin{equation*}
	\Psi(t) = \exp\left(\int_0^tA(\tau,u(\tau)\mathrm d\tau\right)\Psi(0),
\end{equation*}
which yields the simplified step-invariant transform on the group $\G$:
\begin{equation}\label{eq:GlobalACommutes}
	\Psi[k + 1] = \exp\left(\int_0^TA(\tau,u[k])\mathrm d\tau\right)\Psi[k].
\end{equation}

We use the Baker-Campbell-Hausdorff formula to express these dynamics on the Lie algebra:
\begin{equation*}
	\Log(\exp(X)\exp(Y)) = X + Y + \frac{1}{2}[X,Y] + \frac{1}{12}[X,[X,Y]] + \frac{1}{12}[Y,[Y,X]] + \cdots,
\end{equation*}
which yields
\begin{multline*}
	\Log(\Psi[k + 1]) = \Log(\Psi[k]) + \int_0^TA(\tau,u[k])\mathrm d\tau + \frac{1}{2}\left[\int_0^TA(\tau,u[k])\mathrm d\tau,\Log(\Psi[k])\right] \\
		+ \frac{1}{12}\left[\int_0^TA(\tau,u[k])\mathrm d\tau,\left[\int_0^TA(\tau,u[k])\mathrm d\tau,\Log(\Psi[k])\right]\right] \\
		+ \frac{1}{12}\left[\Log(\Psi[k]),\left[\Log(\Psi[k]),\int_0^TA(\tau,u[k])\mathrm d\tau\right]\right] + \cdots,
\end{multline*}
%
which is a linear combination of linear terms and nested Lie brackets.

Exact solutions, and therefore step-invariant transforms, are not unique to right- (or left-) invariant vector fields. For example, the ODE in the variable $X \in \g$,
\begin{equation*}
	\dot X = XA - AX
\end{equation*}
has the closed-form solution~\cite[Proof of Proposition 2.2]{Elliott2009}
\begin{equation*}
	X(t) = e^{\ad_{tA}}X(0),
\end{equation*}
where $e^{\ad_{tA}} := \Id_\g + \ad_{tA} + \frac{1}{2!}\ad_{tA}^2 + \frac{1}{3!}\ad_{tA}^3 + \cdots$, which furnishes the step-invariant transform
\begin{equation*}
	X[k + 1] = e^{\ad_{TA}}X[k] = X[k] + T[A,X[k]] 
		+ \frac{T^2}{2!}[A,[A,X[k]]] + \frac{T^3}{3!}[A,[A,[A,X[k]]]] + \cdots.
\end{equation*}
%


All the sampled dynamics presented in this section are examples of Lie functions, in particular, they belong to class-$\mc A$, which we define in Section~\ref{sec:f} and is the main class of systems studied in this paper.

\subsection{Notation and Terminology}
Given a set $\mc X$, a map $x : \Int \to \mc X$ is a \emph{discrete-time signal}. The notation $x[k]$, with brackets, in contrast to parentheses, implicitly defines the discrete-time signal $x$. The notation $x$ and $x^+$ will often be used as shorthand for $x[k]$ and $x[k + 1]$, respectively, when the time index is clear or irrelevant. All vector spaces encountered are assumed to be finite dimensional. Given a vector space $\mc X$ with subspace $\mc V \subseteq \mc X$, $\mc X / \mc V$ denotes the quotient (or factor) space with cosets $\bar x := \{v \in \mc X : x - v \in \mc V\}$; we will sometimes use the notation $x + \mc V$ for this same coset. If $\mc T$ is a Cartesian product of a vector space $\mc X$ with itself $n$ times, and $\mc V \subseteq \mc X$, we will sometimes use the notation $\mc T / \mc V$ as shorthand for $\mc T / \mc V^n = \mc X^n / \mc V^n = (\mc X / \mc V)^n$. Given a linear endomorphism of vector spaces $A : \mc X \to \mc X$, let $\rho(A)$ denote its spectral radius, and $\|A\|$ denote the operator norm induced by the vector norm $\|\cdot\|$ on $\mc X$; unless stated otherwise, the choice of norm is immaterial. Given vector spaces $\mc X_1,\ldots,\mc X_n$, with respective norms $\|\cdot\|_{\mc X_1},\ldots,\|\cdot\|_{\mc X_n}$, we define the product norm on $\mc X_1 \times \cdots \times \mc X_n$ by $\|(X_1,\ldots,X_n)\| := \sum_{i = 1}^n\|X_i\|_{\mc X_i}$. Given a Lie algebra $\g$, let $[\cdot,\cdot] : \g \times \g \to \g$ denote its Lie bracket. Given two Lie subalgebras $\h_1,\h_2 \subseteq \g$, $[\h_1,\h_2] := \{[H_1,H_2] \in \g : H_1 \in \h_1, H_2 \in \h_2\}$. The symbol $0$ will be used to represent the additive identity on any vector space. A \emph{word} $\omega \in \g$ with \emph{length} $|\omega| \in \Nat$ over the $n \in \Nat$ \emph{letters} $X_1,\ldots,X_n \in \g$ is a (nested) Lie bracket $[X_{\omega_1},[X_{\omega_2},[\ldots X_{\omega_{|\omega|}}]\cdots]$, where $X_{\omega_i} \in \{X_1,\ldots,X_n\}$.

\section{Preliminaries}

We now define what it means for a Lie algebra to be solvable and nilpotent. We also state several algebraic properties of such Lie algebras used in our analysis.

\begin{definition}[Derived Series]
The \textbf{derived series} of a Lie algebra $\g$ is defined recursively by $\g_0 := \g$, $\g_{i + 1} := [\g_i,\g_i]$, for $i \geq 0$ .
\end{definition}

A consequence of the definition of $\g_i$ is that for all $i \geq 0$, $\g_i \supseteq \g_{i + 1}$.

\begin{definition}[Solvable\label{def:Solvable}]
A Lie algebra $\g$ is \textbf{solvable} if there exists a finite $v$ such that $\g_{v + 1} = 0$. The smallest such $v$ is called the \textbf{derived length} of $\g$. A Lie group is solvable if its Lie algebra is solvable.
\end{definition}

If $\g$ is solvable with derived length $v$, then for all $i \leq v$, the containment $\g_i \supset \g_{i + 1}$ is strict.

\begin{definition}[Lower Central Series\label{def:LCS}]
The \textbf{lower central series} of a Lie algebra $\g$ is defined recursively by $\g^{(1)} := \g$, $\g^{(i + 1)} := [\g^{(i)},\g]$, for $i \geq 1$.
\end{definition}

There are two important consequences of Definition~\ref{def:LCS}: the algebras of the lower central series $\g^{(i)}$ are ideals, and for all $i \geq 1$, $\g^{(i)} \supseteq \g^{(i + 1)}$.

\begin{definition}[Nilpotent]
A Lie algebra $\g$ is nilpotent if there exists a finite $p$ such that $\g^{(p + 1)} = 0$. The smallest such $p$ is called the \textbf{nilindex} of $\g$. A Lie group is nilpotent if its Lie algebra is nilpotent.
\end{definition}

The property that serves as the foundation of our analysis, is that if $\g$ is nilpotent, then
\begin{equation*}
	\g^{(1)} \supset \g^{(2)} \supset \cdots \supset \g^{(p)} \supset \g^{(p + 1)} = 0.
\end{equation*}

\begin{theorem}[{\cite[Lemma 1.1.1]{corwin2004representations}\label{thm:StronglyCentral}}]
The ideals of the lower central series of a Lie algebra $\g$ satisfy
%
	$[\g^{(i)},\g^{(j)}] \subseteq \g^{(i + j)}$.
\end{theorem}

Although Definition~\ref{def:Solvable} is the formal definition of solvability, it is the structure endowed by the following theorem that will be leveraged in our analysis.\footnote{If $\h \subseteq \g$ is a nilpotent ideal such that $\h \supseteq [\g,\g]$, then for all $i \geq 2$, $\h^{(i)} \subseteq [\g,\g]^{(i)}$.}

\begin{theorem}[{\cite[p. 9, Corollary 3]{Onishchik1994}}\label{thm:DgNilpotent}]
A Lie algebra $\g$ over $\Complex$ or $\Real$ is solvable if and only if its \textbf{derived algebra} $[\g,\g]$ is nilpotent.
\end{theorem}

In the proofs of our main results, we examine the quotient dynamics on the quotient spaces modulo the ideals of the lower central series. To that end, we require the notion of canonical projection.

\begin{definition}[Canonical Projection]
Let $\mc X$ be a vector space with subspace $\mc V \subseteq \mc X$. The \textbf{canonical projection} of $\mc X$ onto $\mc V$ is the unique linear map $P : \mc X \to \mc X / \mc V$, $x \mapsto x + \mc V$.
\end{definition}

%
%

\begin{proposition}[{\cite[\S$0.7$]{Wonham1979}}\label{prop:QuotientMap}]
Given a linear map $A : \mc X \to \mc X$ and an $A$-invariant subspace $\mc V \subseteq \mc X$, i.e., $A\mc V \subseteq \mc V$, there exists a unique linear map $\bar A : \mc X / \mc V \to \mc X / \mc V$ such that the following diagram commutes.
\begin{equation*}
	\xymatrix{\mc X \ar[d]_P \ar[r]^A 	&	\mc X \ar[d]^P \\
			\mc X / \mc V \ar[r]_{\bar A}	&	\mc X / \mc V}
\end{equation*}
\end{proposition}

The map $\bar A$ in Proposition~\ref{prop:QuotientMap} is called the \emph{map induced in $\mc X / \mc V$ by $A$}, or in short, the \emph{induced map}.

\begin{lemma}\label{lem:CompProj}
Let $\mc X$ be a vector space with subspace $\mc V \subseteq \mc X$. Let $P : \mc X \to \mc X / \mc V$ be the canonical projection, and $\imath : \mc X / \mc V \to \mc X$ be a right-inverse of $P$. Then $(\Id_{\mc X} - \imath \circ P)\mc X \subseteq \mc V$.
\end{lemma}

\begin{proof}
	$P(\Id_{\mc X} - \imath \circ P)
	= P - P \circ \imath \circ P
	= P - P = 0$,
which implies $(\Id_{\mc X} - \imath \circ P)\mc X \subseteq \Ker P$.
\end{proof}


\begin{definition}[Quotient Norm\label{def:QuotientNorm}]
Given a vector space $\mc X$ with norm $\|\cdot\|$ and subspace $\mc V \subseteq \mc X$, if $x \in \mc X$, then the \textbf{quotient norm} of the coset $x + \mc V$ is
\begin{equation*}
	\|x + \mc V\|_{\mc X / \mc V} := \inf_{v \in \mc V}\|x + v\|.
\end{equation*}
\end{definition}

The following result is an obvious consequence of Definition~\ref{def:QuotientNorm}. We formally state it because it is important in the proofs of our main results.

\begin{lemma}\label{lem:QuotientInequality}
Let $\mc X$ be a normed vector space with subspaces $\mc V_1$ and $\mc V_2$, such that $\mc V_1 \subseteq \mc V_2$. For all $x \in \mc X$, we have $\|x + \mc V_2\|_{\mc X / \mc V_2} \leq \|x + \mc V_1\|_{\mc X / \mc V_1} \leq \|x\|$.
\end{lemma}

The following result is elementary, but we state and prove it for completeness, and will use it in our analysis.
\begin{proposition}\label{prop:Pnorm}
Let $\mc X$ be a vector space with norm $\|\cdot\|$, and let $\mc V \subseteq \mc X$ be a subspace. If the quotient norm is used on $\mc X / \mc V$, then the canonical projection $P : \mc X \to \mc X / \mc V$ has unit norm.
\end{proposition}

\begin{proof}
Beginning with the definition of operator norm, we have
\begin{equation*}
\begin{aligned}
	\|P\| :&= \max_{\|x\| = 1}\inf_{v \in \mc V}\|x + v\| \\
	&\leq \max_{\|x\| = 1}\inf_{v \in \mc V}\{\|x\| + \|v\|\} \\
	&= \max_{\|x\| = 1}\|x\| \\
	&= 1,
\end{aligned}
\end{equation*}
which establishes an upper bound of $1$.

Consider a vector $x \in \{x \in \mc X : x \notin \mc V\}$. Then for all $v \in \mc V$
\begin{equation*}
\begin{aligned}
	&& \|Px\|_{\mc X / \mc V} = \|P(x + v)\|_{\mc X / \mc V} &\leq \|P\|\|x + v\| \\
	\implies& & \|Px\|_{\mc X / \mc V} &\leq \|P\|\underbrace{\inf_{v \in \mc V}\|x + v\|}_{\|Px\|_{\mc X / \mc V}} \\
	\iff& &1 &\leq \|P\|.
\end{aligned}
\end{equation*}
Thus, $1 \leq \|P\| \leq 1$, so $\|P\| = 1$.
\end{proof}

%

%

%
\begin{theorem}[{\cite[\S $7$]{Desoer1975}}\label{thm:Anorm}]
Given a linear map $A : \mc X \to \mc X$ and a constant $\varepsilon > 0$, there exists a vector norm $\|\cdot\| : \mc X \to \Real$ such that the induced operator norm satisfies $\|A\| < \rho(A) + \varepsilon$.
\end{theorem}

\begin{remark}\label{rem:Norm}
Given a Lie algebra $\g$ with norm $\|\cdot\|$, there exists $\mu \in [0,2]$, such that for all $X, Y \in \g$, $\|[X,Y]\| \leq \mu\|X\|\|Y\|$.

The lower bound of $0$ holds when $\g$ is commutative, and the upper bound of $2$ is verified by the triangle inequality and submultiplicativity of induced norms:
\begin{equation*}
	\|[X,Y]\| = \|XY - YX\| \leq \|X\|\|Y\| + \|Y\|\|X\| = 2\|X\|\|Y\|.
\end{equation*}

The constant $\mu$ is not necessarily either $0$ or $2$. For example, if $\g$ is any matrix Lie algebra equipped with the Frobenius norm, then $\mu = \sqrt{2}$~\cite[Theorem 2.2]{Bottcher2008}.
\end{remark}

\section{The Class of Systems}\label{sec:f}


\begin{definition}[Lie Element]
Let $X_1,\ldots,X_n$ be the free generators of a Lie algebra $\g$. The elements $X_1,\ldots,X_n$ are called \textbf{Lie elements} of degree one. The Lie brackets $[X_i,X_j]$ are Lie elements of degree two, $[X_i,[X_j,X_k]]$ Lie elements of degree three, and so forth. Any linear combination of Lie elements -- not necessarily finite -- with complex coefficients is also a Lie element.
\end{definition}

\begin{definition}
A function $f : \g^n \to \g$ is a \textbf{Lie function} if there exists open $U \subseteq \g^n$ such that for all $X \in U$, $f(X)$ is a Lie element.
\end{definition}

If $f_1,\ldots,f_m$ are Lie functions, whose scalar coefficients of the word $\omega$ are respectively $c^1_\omega,\ldots,c^m_\omega \in \Field$, where $\Field$ is $\Complex$ or $\Real$, then
\begin{equation*}
	f(X_1,\ldots,X_n) := \begin{bmatrix}f_1(X_1,\ldots,X_n) \\ \vdots \\ f_m(X_1,\ldots,X_n)\end{bmatrix}
	= \begin{bmatrix}\sum_\omega c^1_\omega\omega \\ \vdots \\ \sum_\omega c^m_\omega\omega\end{bmatrix}
	= \sum_\omega\begin{bmatrix}c^1_\omega \\ \vdots \\ c^m_\omega\end{bmatrix} \otimes \omega,
\end{equation*}
which we write compactly as
\begin{equation}\label{eq:fA}
	f(X) = \sum_\omega c_\omega \otimes \omega.
\end{equation}

Given $f : \g^n \to \g$, the following theorem can be used to test whether it is a Lie function.

\begin{theorem}[{Friedrichs' Theorem~\cite[Theorem 1]{Magnus1953}}]
A map $f : \g^n \to \g$ is a Lie function if and only if, for all $X_1,\ldots,X_n,Y_1,\ldots,Y_n \in \g$ such that for all $i \neq j$, $[X_i,Y_j] = 0$,
\begin{equation*}
	f(X_1 + Y_1,\ldots,X_n + Y_n) = f(X_1,\ldots,X_n) + f(Y_1,\ldots,Y_n).
\end{equation*}
\end{theorem}

We consider systems whose dynamical maps are Lie functions, but we also impose that they enjoy a strong form of convergence, as characterized in the following definition.

\begin{definition}[Class-$\mc A$ Function\label{def:Analaytic}]
A Lie function $f : \g^n \to \g$ belongs to \textbf{class}-$\mc A$ -- which we write as $f \in \mc A$ -- if there exists a neighbourhood of the origin in $\g^n$ where the series representation of $f$ satisfies the strong absolute convergence property:
\begin{equation}\label{eq:MyConvergence}
	\sum_\omega \mu^{|\omega| - 1}|c_\omega|\|X_{\omega_1}\|\cdots\|X_{\omega_{|\omega|}}\| < \infty.
\end{equation}

A product map $f_1 \times \cdots \times f_m : \g^{n_1} \times \cdots \times \g^{n_m} \to \g^m$ belongs to class-$\mc A$ if each component map belongs to class-$\mc A$.
\end{definition}

\begin{remark}
Property~\eqref{eq:MyConvergence}, enjoyed by $f \in \mc A$, is stronger than absolute convergence, i.e., $\sum_\omega |c_\omega|\|\omega\| < \infty$, since
%
	$\|\omega\| \leq \mu^{|\omega| - 1}\|X_{\omega_1}\|\cdots\|X_{\omega_{|\omega|}}\|$.
\end{remark}

\begin{remark}
By the Baker-Campbell-Hausdorff formula, we have that the map $\Log(\exp(X)\exp(Y))$ belongs to class-$\mc A$:
\begin{equation}\label{eq:BCH}
	\Log(\exp(X)\exp(Y)) =X + Y + \frac{1}{2}[X,Y] + \frac{1}{12}[X,[X,Y]] + \frac{1}{12}[Y,[Y,X]] + \cdots.
\end{equation}
To see that~\eqref{eq:BCH} satisfies~\eqref{eq:MyConvergence}, refer to~\cite[Proof of Theorem 8]{Day1991} or~\cite{Blanes2004}, and the references therein.
\end{remark}

\begin{remark}
That $\Log(\exp(X)\exp(Y))$ belongs to class-$\mc A$ means that the sampled-data dynamics of a system on a matrix Lie group of the form~\eqref{eq:GlobalACommutes} have local dynamics that are class-$\mc A$, which, as discussed in the Introduction, motivates the study of this class of systems.
\end{remark}

\begin{proposition}\label{prop:NormClassA}
If the product map~\eqref{eq:fA} belongs to class-$\mc A$, then
\begin{equation*}
	\sum_\omega\mu^{|\omega| - 1}\|c_\omega\|\|X_{\omega_1}\|\cdots\|X_{\omega_{|\omega|}}\| < \infty.
\end{equation*}
\end{proposition}

\begin{proof}
By definition, $f \in \mc A$ implies $f_i \in \mc A$, which means that for all $i \in \{1,\ldots,m\}$,
\begin{equation}\label{eq:fiSum}
	\sum_\omega\mu^{|\omega| - 1}|c^i_\omega|\|X_{\omega_1}\|\cdots\|X_{\omega_{|\omega|}}\| < \infty.
\end{equation}

Summing~\eqref{eq:fiSum} over $1 \leq j \leq m$:
\begin{equation*}
	\sum_\omega\mu^{|\omega| - 1}\|c_\omega\|_1\|X_{\omega_1}\|\cdots\|X_{\omega_{|\omega|}}\|< \infty,
\end{equation*}
where $\|\cdot\|_1$ is the $1$-norm. On a finite dimensional vector space, all norms are equivalent, so this summation differs from that in the proposition by at most a constant, finite factor.
\end{proof}

If the Lie algebra $\g$ is nilpotent, then only finitely many words are nonzero; consequently~\eqref{eq:FullPower} trivially satisfies the class-$\mc A$ convergence property~\eqref{eq:MyConvergence} globally. We now impose the major structural assumption on the class of systems~\eqref{eq:sys} under consideration.

\begin{assumption}\label{asm:f}
The function $f : \mc X \times \mc W \to \mc X$ in~\eqref{eq:sys} enjoys the following properties:
\begin{enumerate}
	\item $f$ belongs to class-$\mc A$;\label{asm:ClassA}
	\item \label{asm:vanish} the origin of the state-space $\mc X$ is a unique equilibrium,
		\begin{equation*}
			f(X,\mc W) = 0 \iff X = 0;
		\end{equation*}
	\item \label{asm:invariance} there exists an ideal $\h \subseteq \g$ with nilindex $p$, such that $\h \supseteq [\g,\g]$, each of whose ideals in its lower central series $\left(\h^{(i)}\right)^n \subseteq \mc X$ are invariant under $f$, i.e.,
		\begin{equation*}
			{f\left(\left(\h^{(i)}\right)^n,\mc W\right) \subseteq \left(\h^{(i)}\right)^n}.
		\end{equation*}
\end{enumerate}
\end{assumption}

\begin{remark}
Assumption~\ref{asm:f}.\ref{asm:invariance} may seem restrictive, however, in the context of control theory, it is not unreasonable, because the control signal can be used to enforce invariance. Consider, for example, the step-invariant transform of the driftless kinematics of a fully actuated rigid body with velocity inputs on the solvable Lie group $\SE{2}$:
\begin{equation*}
	X[k + 1] =  \exp\left(T\left(\begin{bmatrix}0 & -1 & 0 \\ 1 & 0 & 0 \\ 0 & 0 & 0\end{bmatrix}u_1[k]
		+ \begin{bmatrix}0 & 0 & 1 \\ 0 & 0 & 0 \\ 0 & 0 & 0\end{bmatrix}u_2[k]
		+ \begin{bmatrix}0 & 0 & 0 \\ 0 & 0 & 1 \\ 0 & 0 & 0\end{bmatrix}u_3[k]\right)\right)X[k],
\end{equation*}
where $X \in \SE{2}$, $u_1,u_2,u_3 \in \Real$, $T > 0$. The inputs $u_1,u_2,u_3$ can be chosen to make any subspace of $\mathfrak{se}(2)$ invariant under the local dynamics. Forthcoming papers by the current authors treat a more general class of systems in the contexts of synchronization and output regulation, and show that they (can be made to) satisfy this dynamical invariance assumption.
\end{remark}

Define the notation $\widetilde X := \{X_1,\ldots,X_n\}$ and $\widetilde W := \{W_1,\ldots,W_r\}$. Henceforth, we adopt the convention that summations over $\omega$ are restricted to words of length at least $2$; words of length $1$ will be written separately, in particular, under Assumption~\ref{asm:f}, the dynamics~\eqref{eq:sys} can be written as
\begin{equation}\label{eq:FullPower}
	f(X,W) = AX + BW + \sum_{\omega}c_\omega \otimes \omega,
\end{equation}
where $A : \mc X \to \mc X$, $B : \mc W \to \mc X$ are linear maps, $\omega$ is a word with letters in $\widetilde X \cup \widetilde W$, and $c_\omega \in \Field^n$  is the vector of coefficients of $\omega$ in the series representation of each component function $f_i$.

\begin{proposition}\label{prop:1X}
If the function $f : \mc X \times \mc W \to \mc X$ in~\eqref{eq:sys} is a Lie function that satisfies Assumption~\ref{asm:f}.\ref{asm:vanish}, then every word in the power series of $f$ has at least one letter in $\widetilde X$.
\end{proposition}

\begin{proof}
By bilinearity of the Lie bracket, all words with at least one letter in $\widetilde X$ vanish at $X = 0$. Setting $X = 0$ in~\eqref{eq:FullPower} yields
\begin{equation}\label{eq:W0}
	0 = BW + \sum_{\omega \text{ with no letters in $\widetilde X$}}c_\omega \otimes \omega,
\end{equation}
which holds for all $W \in \mc W$.
\end{proof}

Therefore, without loss of generality, we can take $B$ and the coefficients of all words $\omega$ with no letters in $\widetilde X$ to be zero. By Proposition~\ref{prop:1X}, henceforth, systems that satisfy Assumption~\ref{asm:f} will be written:
\begin{equation}\label{eq:PowerSeries}
	X^+ = AX + \sum_\omega c_\omega \otimes \omega,
\end{equation}
where every word $\omega$ has at least one letter in $\widetilde X$.

\begin{proposition}\label{prop:Linearization}
If the function $f : \mc X \times \mc W \to \mc X$ in~\eqref{eq:sys} satisfies Assumptions~\ref{asm:f}.\ref{asm:ClassA} and~\ref{asm:f}.\ref{asm:vanish}, then its linearization at the origin, $(X,W) = (0,0) \in \mc X \times \mc W$, is $f(X,W) \approx AX$.
\end{proposition}

\begin{proof}
The Fr\'echet derivative of $f(X,W)$ at the origin in the direction $H := (H_X,H_W) \in \mc X \times \mc W$ is the unique linear map $Df := D_Xf \times D_Wf$ that satisfies
\begin{equation}\label{eq:Df}
	\lim_{H \to 0}\frac{\|f(H_X,H_W) - f(0,0) - DfH\|}{\|H\|} = 0.
\end{equation}

Substituting definitions, and invoking Assumption~\ref{asm:f}.\ref{asm:vanish} and Proposition~\ref{prop:1X} to set $B = 0$, the left side of~\eqref{eq:Df} becomes
\begin{equation*}
	\lim_{H \to 0}\frac{\|(A - D_Xf)H_X + \sum_\omega c_\omega \otimes \omega - D_WfH_W\|}{\|H\|},
\end{equation*}
where the letters of $\omega$ are $H_1,\ldots,H_n$ instead of $X_1,\ldots,X_n$ and $H_{n + 1},\ldots,H_{n + r}$ instead of $W_1,\ldots,W_r$. Suppose $D_Xf = A$ and $D_Wf = 0$, then
\begin{equation*}
	\lim_{H \to 0}\frac{\|f(H_X,H_W) - f(0,0) - DfH\|}{\|H\|} 
	= \lim_{H \to 0}\frac{\|\sum_\omega c_\omega \otimes \omega\|}{\|H\|}.
\end{equation*}

By the result discussed in Remark~\ref{rem:Norm},
\begin{equation*}
	\|\omega\| = 
	 \|[H_{\omega_1},[\ldots,H_{\omega_{|\omega|}}]\cdots]\|
	 \leq \mu^{|\omega| - 1}\|H_{\omega_1}\|\cdots\|H_{\omega_{|\omega|}}\|
	 \leq \mu^{|\omega| - 1}\|H\|^{|\omega|}.
\end{equation*}

By the triangle inequality,
\begin{equation*}
	\left\|\sum_\omega c_\omega \otimes \omega\right\|
	\leq \sum_\omega \|c_\omega\|\mu^{|\omega| - 1}\|H\|^{|\omega|},
\end{equation*}
whose right side converges, by Assumption~\ref{asm:f}.\ref{asm:ClassA}. Therefore,

\begin{equation*}
	\lim_{H \to 0}\frac{\|\sum_\omega c_\omega \otimes \omega\|}{\|H\|}
	\leq \lim_{H \to 0}\frac{\sum_\omega \|c_\omega\|\mu^{|\omega| - 1}\|H\|^{|\omega|}}{\|H\|}
	= 0.
\end{equation*}


%

Since any such $Df$ is unique, the choice of $Df = A \times 0$ is the Fr\'echet derivative of $f$ at the origin. Therefore, near the origin, $f(X,W) \approx AX$.
\end{proof}

Our main results assert that global stability properties of~\eqref{eq:sys} under Assumption~\ref{asm:f} can be inferred from its Jacobian linearization, as quantified in Proposition~\ref{prop:Linearization}. The following proposition asserts that the dynamical invariance described in Assumption~\ref{asm:f}.\ref{asm:invariance} can also be inferred from the Jacobian linearization. This latter result is due to strong centrality of the lower central series, i.e., the property described in Theorem~\ref{thm:StronglyCentral}.

\begin{proposition}\label{prop:Ainvariant}
Let $\h \subseteq \g$ be an ideal. If the function $f : \mc X \times \mc W \to \mc W$ in~\eqref{eq:sys} is a Lie function that satisfies Assumption~\ref{asm:f}.\ref{asm:vanish}, then $f\left(\left(\h^{(i)}\right)^n,\mc W\right) \subseteq \left(\h^{(i)}\right)^n$ if and only if $\left(\h^{(i)}\right)^n$ is invariant under $A$.
\end{proposition}

\begin{proof}
Let $\h \subseteq \g$ be an ideal. Suppose $X \in \left(\h^{(i)}\right)^n$. Under Assumption~\ref{asm:f}.\ref{asm:vanish}, by Proposition~\ref{prop:1X}, every word $\omega$ has at least one letter in $\widetilde X$. Since $\h^{(i)}$ is an ideal, every word $\omega$ belongs to $\h^{(i)}$. From~\eqref{eq:PowerSeries}, we conclude $f\left(\left(\h^{(i)}\right)^n \times \mc W\right) \subseteq \left(\h^{(i)}\right)^n$ if and only if $\left(\h^{(i)}\right)^n$ is invariant under $A$.
\end{proof}

\begin{corollary}\label{cor:Ainvariant}
If the function $f : \mc X \times \mc W \to \mc W$ in~\eqref{eq:sys} is a Lie function that satisfies Assumption~\ref{asm:f}.\ref{asm:vanish}, then it satisfies Assumption~\ref{asm:f}.\ref{asm:invariance} if and only if $\left(\h^{(i)}\right)^n$ is invariant under $A$.
\end{corollary}

Our next result emphasizes that $A$-invariant subspaces induce well-defined quotient systems associated with the nonlinear dynamics.

\begin{proposition}\label{prop:QuotientDynamics}
If the function $f : \mc X \times \mc W \to \mc X$ in~\eqref{eq:sys} satisfies Assumption~\ref{asm:f}, then, given an $A$-invariant ideal $\mc V \subseteq \mc X$ with canonical projection $P : \g \to \g / \mc V$, there exists a unique function $\bar f : \mc X / \mc V \times \mc W / \mc V \to \mc X / \mc V$ that satisfies Assumptions~\ref{asm:f}.\ref{asm:ClassA} and \ref{asm:f}.\ref{asm:vanish}, and makes the following diagram commute.
%
\begin{equation*}
	\xymatrix{\mc X \times \mc W \ar[d]_{(I_n \otimes P) \times (I_r \otimes P)} \ar[rr]^f 	&&	\mc X \ar[d]^{I_n \otimes P} \\
			\mc X / \mc V \times \mc W / \mc V \ar[rr]_{\bar f}	&&	\mc X / \mc V}
\end{equation*}
\end{proposition}

\begin{proof}
Along the path $\mc X \times \mc W \overset{f}{\longrightarrow} \mc X \xrightarrow{I_n \otimes P} \mc X / \mc V$, we have
\begin{equation*}
	(I_n \otimes P)f(X,W) = (I_n \otimes P)AX + (I_n \otimes P)\sum_\omega c_\omega \otimes \omega.
\end{equation*}
By Proposition~\ref{prop:QuotientMap}, there exists a unique map $\bar A : \mc X / \mc V \to \mc X / \mc V$ such that $(I_n \otimes P)A = \bar A(I_n \otimes P)$. Using the property of tensor products that $(M_1 \otimes N_1)(M_2 \otimes N_2) = (M_1M_2) \otimes (N_1N_2)$, the projection of the summation over $\omega$ equals $\sum_\omega c_\omega \otimes (P\omega)$. Then, since the canonical projection of an algebra onto an ideal is a morphism of algebras~\cite[p. $537$]{MacLane1999}\footnote{In~\cite{MacLane1999}, a proof is provided in the context of \emph{graded} algebras, but this additional structure is not used.}, we have
\begin{equation*}
	P\omega = P[Y_{\omega_1},[\ldots,Y_{\omega_{|\omega|}}]\cdots] = [PY_{\omega_1},[\ldots,PY_{\omega_{|\omega|}}]_{\g / \mc V}\cdots]_{\g / \mc V}, \qquad Y_{\omega_i} \in \widetilde X \cup \widetilde W.
\end{equation*}

The map $\bar f : \mc X / \mc V \times \mc W / \mc V \to \mc X / \mc V$ is then given by
\begin{equation*}
	\bar f(\bar X,\bar W) := \bar A \bar X + \sum_\omega c_\omega \otimes [\bar Y_{\omega_1},[\ldots,\bar Y_{\omega_{|\omega|}}]_{\g / \mc V}\cdots]_{\g / \mc V},
\end{equation*}
%
where $\bar Y_{\omega_i} = PY_{\omega_i}$. That $\bar f$ satisfies Assumption~\ref{asm:f}.\ref{asm:ClassA} follows from Lemma~\ref{lem:QuotientInequality}; satisfaction of Assumption~\ref{asm:f}.\ref{asm:vanish} is clear from the definition of $\bar f$.
\end{proof}

\section{Nilpotent Lie Algebras}

In this section, we present a global stability result in the case that $\g$ is nilpotent, and the ideal $\h$ satisfying Assumption~\ref{asm:f}.\ref{asm:invariance} is $\g$ itself. We devote this section to this specific case because, as will be seen, the results are much stronger than in the general case. The general case where Assumption~\ref{asm:f}.\ref{asm:invariance} is satisfied by a proper ideal is addressed in Section~\ref{sec:Solvable}. The stability property proved in this section is \emph{semiglobal-exponential stability}. The following definition is the natural adaptation of a continuous-time definition, taken from~\cite{Loria2005}.

\begin{definition}[{\cite[Definition $2.7$]{Loria2005}}]
Given a discrete-time dynamical system $x^+ = f(k,x)$, $x \in \mc X$, the origin of $\mc X$ is \textbf{semiglobally exponentially stable} if for all $M > 0$, there exist $\alpha \geq 0$, $\lambda < 1$ such that if $\|x[0]\| \leq M$, then for all $k \geq 0$,
\begin{equation*}
	\|x[k]\| \leq \alpha\lambda^k\|x[0]\|.
\end{equation*}
\end{definition}

%

It follows immediately from the definition that semiglobal exponential stability implies local exponential stability. Our main result in the nilpotent case is that a sufficiently small spectral radius of $A$ implies semiglobal exponential stability.

\begin{theorem}\label{thm:Nilpotent}
Let $\g$ be a nilpotent Lie algebra, and define $\mc X := \g^n$ and $\mc W := \g^r$. Consider the dynamics~\eqref{eq:sys} and suppose $f : \mc X \times \mc W \to \mc X$ satisfies Assumption~\ref{asm:f}, where Assumption~\ref{asm:f}.\ref{asm:invariance} is satisfied with $\h = \g$. If there exist $\beta \geq 0$, $s \geq 1$ such that $\|W[k]\| \leq \beta s^k$, and $\rho(A) < s^\frac{p(1 - p)}{2}$, then the origin of $\mc X$ is semiglobally exponentially stable.
\end{theorem}

\begin{remark}
The assertion that $W$ is bounded by a function of the form $\beta s^k$ implies that it is $Z$-transformable.
\end{remark}

Our proof of Theorem~\ref{thm:Nilpotent} makes extensive use of canonical projections of $\g$ onto $\g / \g^{(i + 1)}$, where $\g^{(i + 1)}$ is an ideal of the lower central series of $\g$ (recall Definition~\ref{def:LCS}). Throughout this section, let $P_i : \g \to \g/\g^{(i + 1)}$ denote the canonical projection of $\g$ onto $\g^{(i + 1)}$, and let $\imath_i : \g / \g^{(i + 1)} \to \g$ denote any linear injection such that $P_i \circ \imath_i = \Id_{\g / \g^{(i + 1)}}$. Before proving Theorem~\ref{thm:Nilpotent}, we establish several intermediary results.

\begin{lemma}\label{lem:Pw}
Let $\g$ be a Lie algebra. Given a word $\omega$ with letters $Y_1,\ldots,Y_{|\omega|} \in \g$,
\begin{equation*}
	P_i\omega = P_i [\imath_{i - 1} \circ P_{i - 1}Y_1,[\ldots,\imath_{i - 1} \circ P_{i - 1}Y_{|\omega|}]\cdots].
\end{equation*}
\end{lemma}

\begin{proof}
By bilinearity of the Lie bracket and Lemma~\ref{lem:CompProj},
\begin{equation}\label{eq:NilPw}
	P_i\omega = P_i\underbrace{[\underbrace{(\Id_\g - \imath_{i - 1} \circ P_{i - 1})Y_1}_{\in \g^{(i)}},[Y_2,[\ldots,Y_{|\omega|}]\cdots]}_{\in \g^{(i + 1)}} + P_i[\imath_{i - 1} \circ P_{i - 1}Y_1,[Y_2,[\ldots,Y_{|\omega|}]\cdots],
\end{equation}
where membership in $\g^{(i + 1)}$ follows from the property of the ideals discussed in Theorem~\ref{thm:StronglyCentral}; the first term is zero, since $P_i\g^{(i + 1)} = 0$, by definition of $P_i$. Applying the same decomposition to the rest of the letters yields the result.
\end{proof}

\begin{proof}[Theorem~\ref{thm:Nilpotent}]
Assume that there exist $\beta \geq 0$, $s \geq 1$ such that $\|W[k]\| \leq \beta s^k$, and that $\rho(A) < s^\frac{p(1 - p)}{2}$; the latter implies that $A$ is Schur, since $p,s \geq 1$. Let $M > 0$ be arbitrary and assume $\|X[0]\| \leq M$. We examine the quotient dynamics on $\mc X / \g^{(i + 1)}$ for all $i$. Since $\g$ is nilpotent, the quotient algebra $\g / \g^{(i + 1)}$ is nilpotent with nilindex $i$, thus for all $|\omega| > i$, $P_i\omega = 0$. By Proposition~\ref{prop:QuotientDynamics},
\begin{equation}\label{eq:QuotientDynamics}
	\bar X_i^+ = \bar A_i\bar X_i + \sum_{|\omega| \leq i}c_\omega \otimes (P_i\bar\omega_{i - 1}),
\end{equation}
where $\bar\omega_{i - 1} \in \g$ is the word $\omega$ with $\imath_{i - 1} \circ P_{i - 1}$ applied to each of its letters, per Lemma~\ref{lem:Pw}.

Since $A : \g^n \to \g^n$ is Schur, every induced map $\bar A_i : \left(\g/\g^{(i + 1)}\right)^n \to \left(\g/\g^{(i + 1)}\right)^n$ is also Schur. The quotient dynamics~\eqref{eq:QuotientDynamics} have the form of a linear system with state $\bar X_i$ and exogenous input
\begin{equation}\label{eq:ui}
	u_i := \sum_{|\omega| \leq i}c_\omega \otimes (P_i\bar\omega_{i - 1}),
\end{equation}
which does not depend on $\bar X_i$. Even though quotient state $i - 1$ drives quotient state $i$, the analysis does not exploit a serial structure; rather, each subsequent quotient system is a ``larger piece'' of the full dynamics. We will show that each quotient system is semiglobally exponentially stable. Our proof is by finite induction. The approach is to show that each quotient system is semiglobally exponentially stable, and, since $\g^{(i)} = 0$ for $i > p$, the $p$th quotient system is simply the original system.


Before proceeding, we define some key values. Since $A$ is Schur, for any $\varepsilon \in (0,1 - \rho(A))$, define $\Lambda := \rho(A) + \varepsilon$, then there exists a $\sigma \geq 0$ such that for all $k \geq 0$, $\|A^k\| \leq \sigma\Lambda^k$~\cite[\S $5$]{LaSalle1986}. Define 
\begin{equation*}
	\Lambda_i := \rho(\bar A_i) + \frac{i}{p + 1}\varepsilon, \quad 1 \leq i \leq p,
\end{equation*}
then for all $i$, there exists $\sigma_i \geq 0$ such that $\|\bar A_i^k\| \leq \sigma_i\Lambda_i^k$. Note $\Lambda_1 < \cdots < \Lambda_p < \Lambda < 1$.

We begin with the base case, $i = 1$:
\begin{equation*}
	\bar X_1^+ = \bar A_1\bar X_1,
\end{equation*}
which is an unforced linear time-invariant system. Consequently, $\bar X_1[k] = \bar A_1^k\bar X_1[0]$, so we have $\|\bar X_1[k]\| \leq \sigma_1\Lambda_1^k\|\bar X_1[0]\| \leq \sigma_1\Lambda^k\|\bar X_1[0]\|$. Let $\alpha_1 := \sigma_1$ and $\lambda_1 := \Lambda$.

By way of induction, we assert that there exists $\alpha_{i - 1} > 0$ such that
%
\begin{equation}
	\|\bar X_{i - 1}[k]\| \leq \alpha_{i - 1}\lambda_{i - 1}^k\|\bar X_{i - 1}[0]\| \label{eq:IndHypDecay},
\end{equation}
%
where for $1 \leq i - 1 \leq p - 1$, $\lambda_{i - 1} := \Lambda s^\frac{(i - 1)(i - 2)}{2}$. We remark that $\frac{(i - 1)(i - 2)}{2}$ is the sum of all natural numbers less than $i - 1$. Note also that by Lemma~\ref{lem:QuotientInequality}, $\|X[0]\| \leq M$ implies $\|\bar X_{i - 1}[0]\| \leq M$.

We now prove that case $i - 1$ implies case $i$. Fix $1 \leq j \leq n$ and choose an arbitrary word $\omega$ in the power series of $f_j$. Denote its letters by $Y_k \in \widetilde X \cup \widetilde W$, $k \in \{1,\ldots,|\omega|\}$, and the number of these letters in $\widetilde X$ by $q$. We will show that the projection of each word $P_i\omega$ converges to zero exponentially. Beginning with Lemma~\ref{lem:Pw},
\begin{equation*}
	P_i\omega = P_i [\imath_{i - 1} \circ P_{i - 1}Y_1,[\ldots,\imath_{i - 1} \circ P_{i - 1}Y_{|\omega|}]\cdots],
\end{equation*}
%
then
\begin{equation*}
	\|P_i\omega\| \leq \mu^{|\omega| - 1}\|\imath_{i - 1}\|^{|\omega|}\prod_{j = 1}^{|\omega|}\|P_{i - 1}Y_j\|.
\end{equation*}
We have $\|P_{i - 1}X_j\| \leq \|(I_n \otimes P_{i - 1})X\|$, and Lemma~\ref{lem:QuotientInequality} implies $\|P_{i - 1}W_j\| \leq \|W\|$. Combining these inequalities with the induction hypothesis~\eqref{eq:IndHypDecay} yields
\begin{align}
	\|P_i\bar\omega_{i - 1}[k]\| &\leq \mu^{|\omega| - 1}\|\imath_{i - 1}\|^{|\omega|}\|\bar X_{i - 1}[k]\|^q\|W[k]\|^{|\omega| - q} \notag \\
	&\leq \mu^{|\omega| - 1}\|\imath_{i - 1}\|^{|\omega|}\left(\alpha_{i - 1}\lambda_{i - 1}^k\|\bar X_{i - 1}[0]\|\right)^q(\beta s^k)^{|\omega| - q} \notag \\
	&= \mu^{|\omega| - 1}\|\imath_{i - 1}\|^{|\omega|}\alpha_{i - 1}^q\beta^{|\omega| - q}(\lambda_{i - 1}^qs^{|\omega| - q})^k\|\bar X_{i - 1}[0]\|^q. \label{eq:PwBound1}
\end{align}
Since $\|X[0]\| \leq M$, in~\eqref{eq:PwBound1}, we use Lemma~\ref{lem:QuotientInequality} to upper bound $q-1$ of the factors of $\|\bar X_{i - 1}[0]\|$ by $M$, and the single remaining factor by $\|\bar X_i[0]\|$:
\begin{equation}\label{eq:PwDecay}
	\|P_i\bar\omega_{i - 1}[k]\| \leq \mu^{|\omega| - 1}\|\imath_{i - 1}\|^{|\omega|}\alpha_{i - 1}^q\beta^{|\omega| - q}(\lambda_{i - 1}^qs^{|\omega| - q})^kM^{q - 1}\|\bar X_i[0]\|.
\end{equation}
%

\begin{claim}\label{claim:u}
There exists $\gamma_i \geq 0$ such that the norm of the exogenous input~\eqref{eq:ui} satisfies
\begin{equation*}
	\|u_i[k]\| \leq \gamma_i({\underbrace{\lambda_{i - 1}s^{i - 1}}_{\lambda_i}})^k\|\bar X_i[0]\|.
\end{equation*}
\end{claim}

The proof of Claim~\ref{claim:u} is in Appendix~\ref{app:u}.
Note that even though $\bar X_i$ and $\bar X_{i - 1}$ are both projections of the state $X$, by the induction hypothesis, the trajectory of $\bar X_{i - 1}$ is fixed, i.e., a function of only time. Thus, despite $\bar X_{i - 1}[k]$ partially determining $\bar X_i[k]$, we can view $\bar X_{i - 1}$ in the dynamics of $\bar X_i$ as an exogenous signal.

By linear systems theory, we can express $\bar X_i[k]$ as the sum of a zero-input response $\bar X_i^\mathrm{zi}[k] = \bar A_i^k\bar X_i[0]$ and a zero-state response $\bar X_i^\mathrm{zs}[k] = \sum_{j = 0}^{k - 1}\bar A_i^{j}u_i[k - 1 - j]$. We now bound the zero-state response thus:
\begin{equation*}
\begin{aligned}
	\|\bar X_i^\mathrm{zs}[k]\| &\leq \sum_{j = 0}^{k - 1}\|\bar A_i^j\|\|u_i[k - 1 - j]\| \\
	&\leq \sum_{j = 0}^{k - 1}\sigma_i\Lambda_i^j\gamma_i\lambda_i^{k - 1 - j}\|\bar X_i[0]\| \qquad \text{(by Claim~\ref{claim:u})} \\
	&\leq \sigma_i\gamma_i\lambda_i^{k - 1}\|\bar X_i[0]\|\sum_{j = 0}^\infty\left(\frac{\Lambda_i}{\lambda_i}\right)^j.
\end{aligned}
\end{equation*}
Recall that for all $1 \leq i \leq p$, $\Lambda_i < \Lambda$, and that by the induction hypothesis, $\lambda_i \geq \Lambda$. Therefore, for all $1 \leq i \leq p$, $\lambda_i > \Lambda_i$. Hence,
\begin{equation*}
	\|\bar X_i^\mathrm{zs}[k]\| \leq \frac{\sigma_i\gamma_i}{\lambda_i - \Lambda_i}\lambda_i^k\|\bar X_i[0]\|.
\end{equation*}

Applying the triangle inequality to $\bar X_i[k] = \bar X_i^\mathrm{zi}[k] + \bar X_i^\mathrm{zs}[k]$, we have
\begin{equation*}
\begin{aligned}
	\|\bar X_i[k]\| &\leq \sigma_i\Lambda_i^k\|\bar X_i[0]\| + \frac{\sigma_i\gamma_i}{\lambda_i - \Lambda_i}\lambda_i^k\|\bar X_i[0]\| \\
	&\leq \underbrace{\sigma_i\left(1 + \frac{\gamma_i}{\lambda_i - \Lambda_i}\right)}_{=: \alpha_i}\lambda_i^k\|\bar X_i[0]\|.
\end{aligned}
\end{equation*}
This proves that the origin of $P_i\mc X = \g^n/\left(\g^{(i + 1)}\right)^n$ is semiglobally exponentially stable. This concludes the induction. Recall that $P_{p + j}\g = \g/\g^{(p + j)} = \g/0 \cong \g$, so step $i = p$ of the induction proves that the origin of $\mc X = \g^n$ is semiglobally exponentially stable.
\end{proof}

\begin{corollary}\label{cor:Nilpotent}
Let $\g$ be a nilpotent Lie algebra and $f : \mc X \times \mc W \to \mc X$ satisfy Assumption~\ref{asm:f}, where Assumption~\ref{asm:f}.\ref{asm:invariance} is satisfied with $\h = \g$. If $W$ is bounded, then the origin of $\mc X$ is semiglobally exponentially stable if $A$ is Schur.
\end{corollary}

\begin{proof}
If $W$ is bounded, then $\|W[k]\| \leq \beta s^k$, for $s = 1$ and some finite $\beta$. Apply Theorem~\ref{thm:Nilpotent}.
\end{proof}

\begin{remark}
If $\g$ has nilindex $1$, i.e., $\g$ is commutative, then the dynamics~\eqref{eq:sys} reduce to a linear time-invariant system. The authors exploited this for output regulation and synchronization on commutative matrix Lie groups in~\cite{McCarthy2018} and~\cite{McCarthy2017}, respectively.
\end{remark}

\begin{example}\label{ex:Nilpotent}
In this example, we illustrate the application of Theorem~\ref{thm:Nilpotent} to control design. We will first define a simple regulator problem, then, using Theorem~\ref{thm:Nilpotent}, we will show that the error dynamics are semiglobally exponentially stable.

Let $\g$ be the $3$-dimensional Heisenberg algebra, which is defined by the commutator relations
\begin{equation*}
	[h_1,h_2] = - h_3, \quad [h_1,h_3] = 0, \quad [h_2,h_3] = 0.
\end{equation*}

The lower central series of $\g$ is $\g =: \g^{(1)} \supset \g^{(2)} \supset \g^{(3)} = 0$, where $\g^{(2)} = \mathrm{Lie}_\Real\{h_3\} \cong \mathrm{Span}_\Real\{h_3\}$, thus, $\g$ has nilindex $p = 2$.


Consider the right-invariant dynamical system with state $X \in \G$
\begin{equation*}
	\dot X = (h_1u_1 + h_2u_2 + h_3u_3)X,
\end{equation*}
where $u \in \Real^3$ is the control input. Suppose this system is sampled with period $T = 1$. The step-invariant transform of this system is
\begin{equation}\label{eq:NilpotentX+}
	X^+ = \exp(h_1u_1 + h_2u_2 + h_3u_3)X.
\end{equation}

Suppose we want $X$ to track a reference that is given implicitly by the tracking error
\begin{equation*}
	E = \exp((h_1 + 2h_2 + 3h_3)w)X,
\end{equation*}
where $w \in \Real$ is a known exogenous signal, which evolves according to
\begin{equation}\label{eq:Nilpotentw+}
	w^+ = 2w.
\end{equation}

The goal is to choose $u$ such that $E$ tends to the identity in $\G$. This is equivalent to driving $\Log(E) \in \g$ to $0$, where we express $e := \Log(E)$ in the basis $\{h_1,h_2,h_3\}$:
\begin{equation*}
\begin{aligned}
	\Log(E) &=: e_1h_1 + e_2h_2 + e_3h_3.
\end{aligned}
\end{equation*}

Using~\eqref{eq:NilpotentX+} and the definition of $E$, we find
\begin{equation*}
\begin{aligned}
	E^+ &= \exp(2(h_1 + 2h_2 + 3h_3)w)\exp(h_1u_1 + h_2u_2 + h_3u_3)X \\
	&= \exp(2(h_1 + 2h_2 + 3h_3)w)\exp(h_1u_1 + h_2u_2 + h_3u_3)\exp(-(h_1 + 2h_2 + 3h_3)w)E.
\end{aligned}
\end{equation*}

Using a generalization of the Baker-Campbell-Hausdorff formula~\cite[\S 5]{Day1991}, we express the error dynamics on the Lie algebra:
\begin{equation*}
\begin{aligned}
	e^+ &= 2(h_1 + 2h_2 + 3h_3)w + (h_1u_1 + h_2u_2 + h_3u_3) - (h_1 + 2h_2 + 3h_3)w + e \\
		&\qquad+ \frac{1}{2}[2(h_1 + 2h_2 + 3h_3)w,h_1u_1 + h_2u_2 + h_3u_3] + \frac{1}{2}[2(h_1 + 2h_2 + 3h_3)w,-(h_1 + 2h_2 + 3h_3)w] 
			\\ &\qquad\quad+ \frac{1}{2}[2(h_1 + 2h_2 + 3h_3)w,e] 
		+ \frac{1}{2}[h_1u_1 + h_2u_2 + h_3u_3,-(h_1 + 2h_2 + 3h_3)w] + \frac{1}{2}[h_1u_1 + h_2u_2 + h_3u_3,e] \\
		&\qquad\qquad+ \frac{1}{2}[-(h_1 + 2h_2 + 3h_3)w,e] \\
	&= (w + u_1)h_1 + (2w + u_2)h_2 + (3w + u_3)h_3 + e \\
		&\qquad+ \frac{1}{2}[(w + u_1)h_1 + (2w + u_2)h_2 + (3w + u_3)h_3,e] - \frac{3}{2}[h_1u_1 + h_2u_2 + h_3u_3,\underbrace{(h_1 + 2h_2 + 3h_3)w}_{=: W}].
\end{aligned}
\end{equation*}

The independent signal $W$ evolves according to
\begin{equation*}
\begin{aligned}
	W^+ &= (h_1 + 2h_2 + 3h_3)w^+ \\
	&= 2(h_1 + 2h_2 + 3h_3)w \\
	&= 2W,
\end{aligned}
\end{equation*}
which yields
\begin{equation*}
\begin{aligned}
	W[k] &= 2^kW[0] \\
	\|W[k]\| &= 2^k\|W[0]\|.
\end{aligned}
\end{equation*}
Thus, setting $\beta = \|W[0]\|$ and $s = 2$, we have $\|W[k]\| \leq \beta s^k$.

To apply Theorem~\ref{thm:Nilpotent} to the dynamics of $e$, we must choose the control law $u$ such that Assumption~\ref{asm:f} is satisfied, and the linear part of~\eqref{eq:Nilpotente+} has spectral radius smaller than $s^{-1} = \frac{1}{2}$. After choosing our control law $u$, we will verify that each of Assumptions~\ref{asm:f}.\ref{asm:ClassA}, \ref{asm:f}.\ref{asm:vanish}, and~\ref{asm:f}.\ref{asm:invariance} are satisfied. Per Proposition~\ref{prop:1X}, Assumption~\ref{asm:f}.\ref{asm:vanish} is satisfied only if the linear part of the dynamics does not depend on $W$. This observation, in part, motivates the control law
\begin{equation*}
	u = \begin{bmatrix}-0.75 & 0.25 & 0 \\ -0.25 & -0.75 & 0 \\ 0 & 0 & -0.99\end{bmatrix}e - \begin{bmatrix}1 \\ 2 \\ 3\end{bmatrix}w.
\end{equation*}

Substituting into the dynamics of $e$, we obtain
\begin{multline}\label{eq:Nilpotente+}
	e^+ = (0.25e_1 + 0.25e_2)h_1 + (-0.25e_1 + 0.25e_2)h_2 + (0.01e_3)h_3  \\
		+ \frac{1}{2}[(0.25e_1 + 0.25e_2)h_1 + (-0.25e_1 + 0.25e_2)h_2,e_1h_1 + e_2h_2] \\
		- \frac{3}{2}[(0.25e_1 + 0.25e_2)h_1 + (-0.25e_1 + 0.25e_2)h_2,(h_1 + 2h_2)w].
\end{multline}
%

The dynamics~\eqref{eq:Nilpotente+} are of the form
\begin{equation*}
	e^+ = Ae + \sum_{|\omega| = 2}c_\omega\omega,
\end{equation*}
where in the basis $\g = \mathrm{Lie}_\Real\{h_1,h_2,h_3\} \cong \mathrm{Span}_\Real\{h_1,h_2,h_3\}$, $A : \g \to \g$ has matrix representation
\begin{equation}\label{eq:MatA}
	\mathrm{Mat}A = \begin{bmatrix}0.25 & 0.25 & 0 \\ -0.25 & 0.25 & 0 \\ 0 & 0 & 0.01\end{bmatrix},
\end{equation}

We now verify that~\eqref{eq:Nilpotente+} satisfies Assumption~\ref{asm:f}. By the form of~\eqref{eq:Nilpotente+} and nilpotency of $\g$, the dynamics of $e$ are clearly class-$\mc A$, thus Assumption~\ref{asm:f}.\ref{asm:ClassA} is satisfied.

That $e = 0$ is an equilibrium is verified by substituting $e = 0$ into~\eqref{eq:Nilpotente+}. To verify that $e = 0$ is the only equilibrium, note that by the definition of the Lie bracket on $\g$, the bracket terms in~\eqref{eq:Nilpotente+} lie in $\mathrm{Span}_\Real\{h_3\}$. Therefore, a point $e$ is an equilibrium only if
\begin{equation*}
	\begin{bmatrix}e_1 \\ e_2\end{bmatrix} = \begin{bmatrix}0.25 & 0.25 \\ -0.25 & 0.25\end{bmatrix}\begin{bmatrix}e_1 \\ e_2\end{bmatrix},
\end{equation*}
which holds if and only if $e_1 = e_2 = 0$. If $e_1 = e_2 = 0$, then~\eqref{eq:Nilpotente+} reduces to $e[k + 1] = 0.01e_3h_3$, whose only equilibrium is $e_3 = 0$. This verifies Assumption~\ref{asm:f}.\ref{asm:vanish}.

The block diagonal structure of~\eqref{eq:MatA} makes it clear that $\g^{(2)} = \mathrm{Lie}_\Real\{h_3\} \cong \mathrm{Span}_\Real\{h_3\}$ is invariant. By Corollary~\ref{cor:Ainvariant}, this verifies Assumption~\ref{asm:f}.\ref{asm:invariance}.
By Theorem~\ref{thm:Nilpotent}, $e = 0$ is semiglobally exponentially stable if $\rho(A) < s^{-1} = \frac{1}{2}$. The eigenvalues of~\eqref{eq:MatA} are $\{-0.25 + i0.25, -0.25 - i0.25, 0.01\}$, thus $\rho(A) = \frac{1}{2\sqrt{2}}$. Therefore, $e = 0$ is semiglobally exponentially stable.
We simulate the dynamics of the tracking error using the initial conditions $e[0] = 3h_1 + 2h_2 - h_3$, $w[0] = 1$. The trajectory of $e$ is in Figure~\ref{fig:Nilpotent}. As can be seen, $e$ tends to $0$.
\begin{figure}[htb!]
\centering
\includegraphics[width=1\columnwidth]{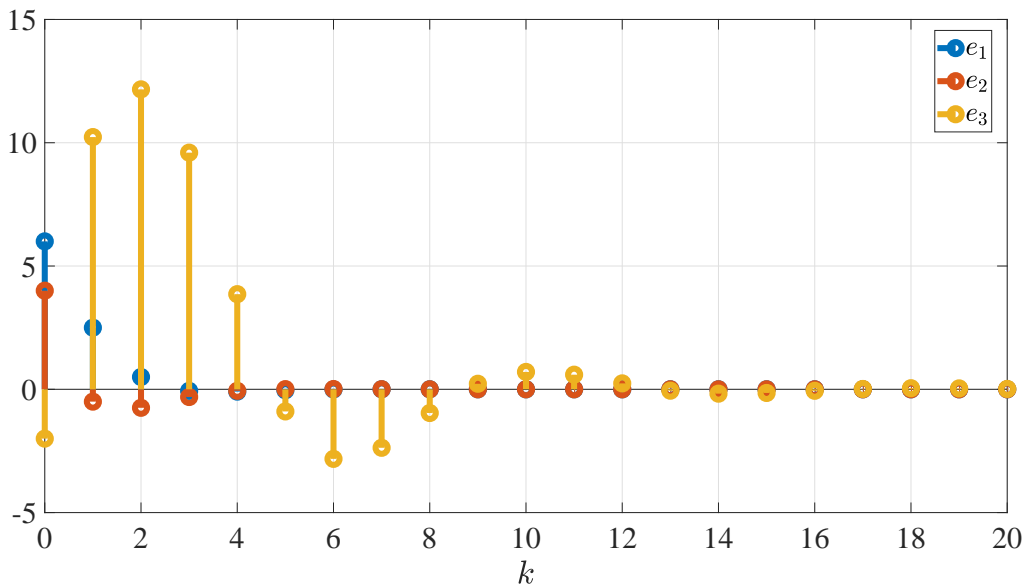}
\caption{The tracking error $e \in \g$ at the sampling instants.}
\label{fig:Nilpotent}
\end{figure}
\exsymbol\end{example}

\section{Solvable Lie Algebras}
\label{sec:Solvable}

In this section we present various global stability results in the case that $\g$ is solvable, but not necessarily nilpotent. Our analysis exploits the structure endowed by Theorem~\ref{thm:DgNilpotent}.

\begin{theorem}\label{thm:Solvable}
Let $\g$ be a solvable Lie algebra, and define $\mc X := \g^n$ and $\mc W := \g^r$. Consider the dynamics~\eqref{eq:sys} and suppose $f : \mc X \times \mc W \to \mc X$ satisfies Assumption~\ref{asm:f}. If $A$ is Schur, and as $k \to \infty$, $W[k] \to \h^r$, then there exists $\beta > 0$ such that if $\limsup_{k \to \infty}\|W[k]\| \leq \beta$, then the origin of $\mc X$ is globally attractive.
\end{theorem}

Theorem~\ref{thm:Solvable} is somewhat weaker than Theorem~\ref{thm:Nilpotent} for the nilpotent case. Although Theorem~\ref{thm:Solvable} would of course apply when the Lie algebra is nilpotent, Theorem~\ref{thm:Nilpotent} is not a special case of Theorem~\ref{thm:Solvable}. The proof of Theorem~\ref{thm:Solvable} takes a similar geometric approach to that of Theorem~\ref{thm:Nilpotent}, but the analysis is significantly complicated by the nontrivial quotient space $\mf K := \g / \h$. The dynamics on $\mf K$ will be treated from an analysis perspective, rather than using geometric arguments, and be shown to converge to the origin via contradiction. Throughout this section, let $P_i : \g \to \g/\h^{(i + 1)} \cong \mf K \oplus \h / \h^{(i + 1)}$ denote the canonical projection of $\g$ onto $\h^{(i + 1)}$. We will require the following lemma, which is the solvable analogue of Lemma~\ref{lem:Pw} in the nilpotent case.

\begin{lemma}\label{lem:SolvBracket}
Let $\g$ be a solvable Lie algebra. Then, given a word $\omega$ with letters $Y_1,\ldots,Y_{|\omega|}$,
%
\begin{equation*}
\begin{aligned}
	P_i\omega &= P_i [\imath_{i - 1} \circ P_{i - 1}Y_1,[\ldots,\imath_{i - 1} \circ P_{i - 1}Y_{|\omega|}]\cdots] \\
		&\qquad+ P_i[\overbrace{(\Id_\g - \imath_{i - 1} \circ P_{i - 1})}^{\text{$1$st letter}}Y_1,[\imath_0 \circ P_0 Y_2,[\ldots,\imath_0 \circ P_0 Y_{|\omega|}]\cdots] \\
		&\qquad\quad+ P_i[\imath_0 \circ P_0 Y_1,[\overbrace{(\Id_\g - \imath_{i - 1} \circ P_{i - 1})}^{\text{$2$nd letter}}Y_2,[\imath_0 \circ P_0 Y_3,[\ldots,\imath_0 \circ P_0 Y_{|\omega|}]\cdots] + \cdots \\
		&\qquad\quad\quad \cdots + P_i[\imath_0 \circ P_0 Y_1,[\ldots,[\imath_0 \circ P_0 Y_{|\omega| - 1},\underbrace{(\Id_\g - \imath_{i - 1} \circ P_{i - 1})}_{\text{$|\omega|$th letter}}Y_{|\omega|}]\cdots].
\end{aligned}
\end{equation*}
\end{lemma}
The proof of Lemma~\ref{lem:SolvBracket} is in Appendix~\ref{app:SolvBracket}.

\begin{proof}[Theorem~\ref{thm:Solvable}]
Analogous to the proof of Theorem~\ref{thm:Nilpotent}, we will examine the quotient dynamics on $\mc X / \h^{(i + 1)}$, where $i \geq 0$. By Proposition~\ref{prop:QuotientDynamics}, the quotient dynamics on $\mc X / \h^{(i + 1)}$ are
\begin{equation}\label{eq:SolvQuotientDynamics}
	\bar X_i^+ = \bar A_i\bar X_i + \sum_\omega c_\omega \otimes (P_i\omega).
\end{equation}

We begin by examining the quotient dynamics on $\mc X / \h = \mf K^n$:
\begin{equation}\label{eq:SolvX0}
	\bar X_0^+ = \bar A_0\bar X_0,
\end{equation}
which is an unforced linear time-invariant system. That $A$ is Schur implies $\bar A_0$ is Schur, so the origin of $P_0\mc X = \g^n / \h^n \cong \mf K^n$ is globally exponentially stable under the quotient dynamics~\eqref{eq:SolvQuotientDynamics}.

We assert the induction hypothesis that the origin of $P_{i - 1}\mc X \cong \mf K^n \oplus \left(\h / \h^{(i)}\right)^n$ is globally asymptotically stable. We now show that the origin of $P_i\mc X \cong \mf K^n \oplus (\h / \h^{(i + 1)})^n$ is globally asymptotically stable.

By Lemma~\ref{lem:SolvBracket},
\begin{multline*}
	P_i\omega = P_i \overbrace{[\imath_{i - 1} \circ P_{i - 1}Y_1,[\ldots,\imath_{i - 1} \circ P_{i - 1}Y_{|\omega|}]\cdots]}^{\hat\omega_{i - 1} :=} \\
		+ P_i[(\Id_\g - \imath_{i - 1} \circ P_{i - 1})Y_1,[\imath_0 \circ P_0 Y_2,[\ldots,\imath_0 \circ P_0 Y_{|\omega|}]\cdots] + \cdots \\
		\cdots + P_i[\imath_0 \circ P_0 Y_1,[\ldots,[\imath_0 \circ P_0 Y_{|\omega| - 1},(\Id_\g - \imath_{i - 1} \circ P_{i - 1})Y_{|\omega|}]\cdots].
\end{multline*}

By the induction hypothesis, each term $P_{i - 1}Y_j$ in $\hat \omega_{i - 1}$ tends to zero, which implies $\hat\omega_{i - 1} \to 0$. We now show $P_i\omega \to 0$. By the result discussed in Remark~\ref{rem:Norm}, Lemma~\ref{lem:SolvBracket}, and that $P_i$ is a morphism of algebras, the norm of each projected word can be bounded thus
\begin{equation}\label{eq:PiwBound}
	\|P_i\omega\| \leq \|\hat\omega_{i -1}\| + \mu^{|\omega| - 1}\sum_{j = 1}^{|\omega|}\left(\|P_i \circ (\Id_\g - \imath_{i - 1} \circ P_{i - 1})Y_j\|\prod_{\ell \neq j}\|P_i \circ \imath_0 \circ P_0Y_\ell\|\right).
\end{equation}

By submultiplicativity of operator norms and Proposition~\ref{prop:Pnorm}, we have 
\begin{equation}\label{eq:i0}
	\|P_i \circ \imath_0 \circ P_0 Y_j\| \leq \|\imath_0 \circ P_0 Y_j\| \leq \|\imath_0\|\|P_0 Y_j\|.
\end{equation}
By Proposition~\ref{prop:Pnorm} and the triangle inequality, we have
\begin{equation}\label{eq:Comp}
	\|(P_i - P_i \circ \imath_{i - 1} \circ P_{i - 1})Y_j\| \leq \|P_iY_j\| + \|\imath_{i - 1}\|\|P_{i - 1}Y_j\| \leq (1 + \|\imath_{i - 1}\|)\|P_iY_j\|,
\end{equation}
where the second inequality follows from Lemma~\ref{lem:QuotientInequality}. We partition the words into the sets $\Omega_X := \{\omega : \text{every letter is in } \widetilde X\}$ and $\Omega_W := \{\omega : \text{at least one letter is in } \widetilde W\}$. First consider $\omega \in \Omega_X$. Applying~\eqref{eq:i0} and~\eqref{eq:Comp} to~\eqref{eq:PiwBound}, we obtain
\begin{equation}\label{eq:PiOmegaXBound}
\begin{aligned}
	\|P_i\omega\| &\leq \|\hat\omega_{i -1}\| + (\mu\|\imath_0\|)^{|\omega| - 1}(1 + \|\imath_{i - 1}\|)\|\bar X_i\|\sum_{j = 1}^{|\omega|}\prod_{\ell \neq j}\|P_0Y_\ell\| \\
	&\leq \|\hat\omega_{i -1}\| + (\mu\|\imath_0\|)^{|\omega| - 1}|\omega|(1 + \|\imath_{i - 1}\|)\|\bar X_0\|^{|\omega| - 1}\|\bar X_i\|,
\end{aligned}
\end{equation}
where we have used $\|P_iX_j\| \leq \|(I_n \otimes P_i) X\| = \|\bar X_i\|$, for all $j \in \{1,\ldots,n\}$.

Now consider $\omega \in \Omega_W$ and let $1 \leq q \leq |\omega| - 1$ be the number of letters in $\widetilde X$. Without loss of generality, suppose $Y_1,\ldots,Y_q \in \widetilde X$, and $Y_{q + 1},\ldots,Y_{|\omega|} \in \widetilde W$. Then
\begin{equation*}
	\sum_{j = 1}^q\prod_{\ell \neq j}\|P_0Y_\ell\| \leq (\mu\|\imath_0\|)^{|\omega| - 1}(1 + \|\imath_{i - 1}\|)\|\bar X_i\|q\|\bar X_0\|^{q - 1}\|\bar W_0\|^{|\omega| - q}
\end{equation*}
and
\begin{equation*}
	\sum_{j = q + 1}^{|\omega|}\prod_{\ell \neq j}\|P_0Y_\ell\| \leq (\mu\|\imath_0\|)^{|\omega| - 1}(1 + \|\imath_{i - 1}\|)\|W\|(|\omega| - q)\|\bar X_0\|^q\|\bar W_0\|^{|\omega| - q - 1}.
\end{equation*}
Using the bounds $q, |\omega| - q \leq |\omega| - 1$, we have
\begin{equation}\label{eq:PiOmegaWBound}
	\|P_i\omega\| \leq \|\hat\omega_{i -1}\| \\ + (\mu\|\imath_0\|)^{|\omega| - 1}(1 + \|\imath_{i - 1}\|)(\|\bar X_i\| + \|W\|)(|\omega| - 1)\max\{\|\bar X_0\|,\|\bar W_0\|\}^{|\omega| - 1}.
\end{equation}
%

Using~\eqref{eq:PiOmegaXBound} and~\eqref{eq:PiOmegaWBound}, we upper bound $\|\bar X_i^+\|$:
\begin{equation*}
\begin{aligned}
	\|\bar X_i^+\| &\leq \|\bar A_i\|\|\bar X_i\| 
		+ \sum_{\omega \in \Omega_W}\|c_\omega\|\|P_i\omega\| 
		+ \sum_{\omega \in \Omega_X}\|c_\omega\|\|P_i\omega\| \\
	&\leq \|\bar A_i\|\|\bar X_i\| 
		+ \sum_{\omega \in \Omega_W \cup \Omega_X}\|c_\omega\|\|\hat\omega_{i -1}\| \\
		&\quad\enspace+ (1 + \|\imath_{i - 1}\|)\|\bar X_i\|\sum_{\omega \in \Omega_X}|\omega|(\mu\|\imath_0\|)^{|\omega| - 1}\|c_\omega\|\|\bar X_0\|^{|\omega| - 1} \\
		&\qquad+ (1 + \|\imath_{i - 1}\|)(\|\bar X_i\| + \|W\|)\sum_{\omega \in \Omega_W}(|\omega| - 1)(\mu\|\imath_0\|)^{|\omega| - 1}\|c_\omega\|\max\{\|\bar X_0\|,\|\bar W_0\|\}^{|\omega| - 1} \\
	&\leq \|\bar A_i\|\|\bar X_i\| 
		+ \sum_\omega\|c_\omega\|\|\hat\omega_{i -1}\| \\
		&\qquad + (1 + \|\imath_{i - 1}\|)(2\|\bar X_i\| + \|W\|)\sum_\omega|\omega|(\mu\|\imath_0\|)^{|\omega| - 1}\|c_\omega\|\max\{\|\bar X_0\|,\|\bar W_0\|\}^{|\omega| - 1}.
\end{aligned}
\end{equation*}


\begin{claim}\label{claim:Converges}
There exists $\varrho > 0$ such that for all $\|\bar X_0\|,\|\bar W_0\| < \varrho$,
\begin{equation}\label{eq:OmegaSeries}
	\sum_\omega|\omega|(\mu\|\imath_0\|)^{|\omega| - 1}\|c_\omega\|\max\{\|\bar X_0\|,\|\bar W_0\|\}^{|\omega| - 1} < \infty.
\end{equation}
\end{claim}


The proof of Claim~\ref{claim:Converges} is in Appendix~\ref{app:Converges}. First, note that the hypothesis $W[k] \to \h^r$ implies $\bar W_0 \to 0$. Now, since~\eqref{eq:OmegaSeries} converges for $\|\bar X_0\|,\|\bar W_0\|$ sufficiently small, it follows that since $\bar X_0$ and $\bar W_0$ tend to zero as $k \to \infty$, that~\eqref{eq:OmegaSeries} tends to zero.

We divide both sides by $\|\bar X_i\|$ and upper bound the limiting supremum thus
%
\begin{multline*}
	\limsup_{k \to \infty}\frac{\|\bar X_i^+\|}{\|\bar X_i\|}
		\leq \|\bar A_i\| \\
		+ \frac{1}{\liminf_{k \to \infty}\|\bar X_i\|}
			\limsup_{k \to \infty}\sum_\omega\|c_\omega\| \Big(
		\|\hat\omega_{i -1}\| 
		\left.+ (1 + \|\imath_{i - 1}\|)\|W\||\omega|(\mu\|\imath_0\|)^{|\omega| - 1}\max\{\|\bar X_0\|,\|\bar W_0\|\}^{|\omega| - 1}\right).
\end{multline*}

Suppose, by way of contradiction, that $\liminf_{k \to\infty}\|\bar X_i\| > 0$. Since $\hat w_{i - 1} \to 0$ and $\bar W_0 \to 0$ by hypothesis, $W$ is bounded, and $\bar X_0 \to 0$, the limiting supremum on the right side is $0$, so
\begin{equation}\label{eq:limsup}
	\limsup_{k \to \infty}\frac{\|\bar X_i^+\|}{\|\bar X_i\|} \leq \|\bar A_i\|.
\end{equation}
All our analysis heretofore has been independent of a specific choice of norm. However, at this point, we invoke Theorem~\ref{thm:Anorm} and choose the norm $\|\cdot\| : \g \to \Real$ such that for some $\varepsilon \in (0,1 - \rho(\bar A_i))$, $\|\bar A_i\| = \rho(\bar A_i) + \varepsilon < 1$. By~\eqref{eq:limsup}, we have $\lim_{k \to \infty}\|\bar X_i\| = 0$, which is a contradiction\footnote{It is merely a coincidence that the contradiction here is the main result we are attempting to prove.}. Therefore, $\liminf_{k \to\infty}\|\bar X_i\| = 0$, so given any $\varepsilon > 0$, there exists a time $k_\varepsilon$ such that $\|\bar X_i[k_\varepsilon]\| < \varepsilon$. By Proposition~\ref{prop:Linearization}, $A$ Schur and $W = 0$ implies local exponential stability of the origin, so by a standard perturbation argument, for $W$ sufficiently small, the origin remains locally exponentially stable. Thus, there exist $\beta > 0$, $\bar k \geq 0$ such that if for all $k \geq \bar k$, $\|W[k]\| \leq \beta$, then the origin of $\mc X$ is locally attractive. Therefore, $\bar X_i$ eventually enters the basin of attraction, so $\bar X_i \to 0$. This establishes that the origin is globally attractive. This proves the induction.
\end{proof}

\begin{remark}
Since the dynamics on $\mc X / \h$ are linear, it could be argued that $[\g,\g]$ is the ``best'' possibility for $\h$, since this maximizes the dimension of $\mc X / \h$. However, the choice of $\h$ does not change the analysis or results.
\end{remark}

If we assert that $W$ is bounded, rather than ultimately bounded, then we can strengthen the attractivity result of Theorem~\ref{thm:Solvable} to stability.

\begin{corollary}\label{cor:SolvGAS}
Let $\g$ be a solvable Lie algebra, and define $\mc X := \g^n$ and $\mc W := \g^r$. Consider the dynamics~\eqref{eq:sys} and suppose $f : \mc X \times \mc W \to \mc X$ satisfies Assumption~\ref{asm:f}. If $A$ is Schur, and as $k \to \infty$, $W[k] \to \h^r$, then there exists $\beta > 0$ such that if $\|W[k]\| \leq \beta$, then the origin of $\mc X$ is globally asymptotically stable.
\end{corollary}

\begin{proof}
The proof is the same as that of Theorem~\ref{thm:Solvable}, where $\bar k = 0$ (defined near the end of the proof of Theorem~\ref{thm:Solvable}), which implies that the origin of $\mc X$ is locally exponentially stable for all $k \geq 0$.
\end{proof}

The requirement that $W$ be indeterminately small in Theorem~\ref{thm:Solvable} and Corollary~\ref{cor:SolvGAS} is rather restrictive. However, when the map $A$ has spectral radius $0$, $W$ need not be bounded, and we can even relax the assumption that $f$ belongs to class-$\mc A$.

\begin{theorem}\label{thm:Deadbeat}
Consider the dynamics~\eqref{eq:sys}. Let $\g$ be a solvable Lie algebra and $f : \mc X \times \mc W \to \mc X$ be a Lie function that satisfies Assumptions~\ref{asm:f}.\ref{asm:vanish} and~\ref{asm:f}.\ref{asm:invariance}. If $\rho(A) = 0$ and for all $k \geq 0$, $W[k] \in \h^r$, then $X$ converges to zero in finite time.
\end{theorem}

\begin{proof}
The quotient dynamics on $\mc X / \h = \mf K^n$ are
\begin{equation*}
	\bar X_0^+ = \bar A_0\bar X_0.
\end{equation*}

That $A$ has spectral radius zero implies that $\bar A_0 : \mf K^n \to \mf K^n$ has spectral radius zero, which implies $\bar A_0^{\dim \mf K} = 0$. Therefore, for all $k \geq \dim \mf K$, we have $\bar X_0[k] = 0$.

By way of induction, we assert that for all $k \geq i\dim \g - \sum_{j = 1}^i\dim \h^{(j)}$, $\bar X_{i - 1}[k] = 0$.

Define $\hat\omega_{i - 1}$, $q$, $\Omega_X$, and $\Omega_W$ as in the proof of Theorem~\ref{thm:Solvable}. If $\omega \in \Omega_X$, then from~\eqref{eq:PiOmegaXBound}, for all $k \geq \dim \mf K$, $\|P_i\omega\| \leq \|\hat\omega_{i - 1}\|$. Since $\|\bar W_0\| = 0$, if $\omega \in \Omega_W$, then from~\eqref{eq:PiOmegaWBound},
\begin{equation*}
	\|P_i\omega\| \leq \|\hat\omega_{i -1}\| + (1 + \|\imath_{i - 1}\|)(\mu\|\imath_0\|)^{|\omega| - 1}\|W\|\|\bar X_0\|^{|\omega| - 1},
\end{equation*}
which for $k \geq \dim \mf K$, simplifies to $\|P_i\omega\| \leq \|\hat\omega_{i - 1}\|$. Since every word $\omega$ has at least one letter in $\widetilde X$, the induction hypothesis implies $\hat\omega_{i - 1} = 0$ for all $k \geq i\dim \g - \sum_{j = 1}^i\dim \h^{(j)}$. Therefore, for all $k \geq i\dim \g - \sum_{j = 1}^i\dim \h^{(j)}$, the quotient dynamics reduce to
\begin{equation*}
	\bar X_i^+ = \bar A_i\bar X_i,
\end{equation*}
where $\rho(\bar A_i) = 0$, and so $\bar A_i^{\dim\left(\g / \h^{(i + 1)}\right)} = 0$, where $\dim\left(\g / \h^{(i + 1)}\right) = \dim \g - \dim \h^{(i + 1)}$; in particular, $\dim \mf K = \dim \g - \dim \h$. Thus, for all $k \geq (i + 1)\dim \g - \sum_{j = 1}^{i + 1}\dim \h^{(j)}$, $\bar X_i[k]$ is zero.

Since $p$ is the nilindex of $\h$, we have $P_p\g = \g / \h^{(p + 1)} = \g / 0 \cong \g$, and so the induction terminates at $i = p$. Consequently, for all $k \geq (p + 1)\dim \g - \sum_{j = 1}^p\dim \h^{(j)}$, $X[k] = 0$.
\end{proof}

\begin{corollary}\label{cor:SolvGA}
Consider the dynamics~\eqref{eq:sys}. Let $\g$ be a solvable Lie algebra and $f : \mc X \times \mc W \to \mc X$ be a Lie function that satisfies Assumption~\ref{asm:f}.\ref{asm:vanish} and~\ref{asm:f}.\ref{asm:invariance}. If $\rho(A) = 0$ and for all $k \geq 0$, $W[k] \in \h^r$, then the origin of $\mc X$ is globally attractive.
\end{corollary}

\begin{proof}
By Theorem~\ref{thm:Deadbeat}, the state $X$ tends to the origin for any initial conditions.
\end{proof}

\begin{corollary}\label{cor:SolvSGES}
Consider the dynamics~\eqref{eq:sys}. Let $\g$ be a solvable Lie algebra and $f : \mc X \times \mc W \to \mc X$ be a Lie function that satisfies Assumption~\ref{asm:f}.\ref{asm:vanish} and~\ref{asm:f}.\ref{asm:invariance}. If $\rho(A) = 0$, there exists $\beta \geq 0$ such that $\|W\| \leq \beta$, and for all $k \geq 0$, $W[k] \in \h^r$, then the origin of $\mc X$ is semiglobally exponentially stable.
\end{corollary}

\begin{proof}
By Theorem~\ref{thm:Deadbeat}, $X[k]$ converges to zero in finite time. Define $\bar k := \argmin_k\{X[k] = 0\}$ and let $M \geq 0$ be arbitrary. Since $\|\cdot\| : \mc X \to \Real$ is continuous, $\|X[k]\|$ attains its maximum on the compact set $\{X[k] : 0 \leq k \leq \bar k,\ \|W[k]\| \leq \beta,\ \|X[0]\| \leq M\}$. Choosing any $\lambda \in [0,1)$, there exists finite $\alpha > 0$ such that $\|X[k]\| \leq \alpha\lambda^k\|X[0]\|$, where $\alpha$ depends on $\|X[0]\|$ and $\beta$.
\end{proof}

\begin{remark}
Theorem~\ref{thm:Deadbeat} and Corollaries~\ref{cor:SolvGA} and~\ref{cor:SolvSGES} easily extend to the case where there exists $k_\h \in \Int_{\geq 0}$ such that for all $k \geq k_\h$, $W[k] \in \h^r$, but $W[0]$ is not necessarily in $\h^r$.
\end{remark}

\begin{example}\label{ex:Solvable}
Consider the $6$-dimensional real upper triangular algebra, whose nonvanishing Lie brackets are
\begin{equation*}
	[t_1,t_4] = t_4, \quad [t_1,t_6] = t_6, \quad
	[t_2,t_4] = -t_4, \quad [t_2,t_5] = t_5, \quad
	[t_3,t_5] = -t_5, \quad [t_3,t_6] = -t_6, \quad
	[t_4,t_5] = t_6.
\end{equation*}

The derived algebra is $\h = \mathrm{Lie}_\Real\{t_4,t_5,t_6\}$, which has lower central series $\h =: \h^{(1)} \supset \h^{(2)} \supset \h^{(3)} = 0$, where $\h^{(2)} = \mathrm{Lie}_\Real\{h_6\} \cong \mathrm{Span}_\Real\{h_6\}$. We remark that the derived algebra $\h$ and the Heisenberg algebra are isomorphic as Lie algebras.

We will consider a dynamical system driven by the exogenous signal $W := (W_1,W_2) \in \g^2 =: \mc W$
\begin{equation*}
\begin{aligned}
	W_1^+ &= 2\left(1 - k(1.1)^{-0.5k}\right)\sin(10k)W_0 \\
	W_2^+ &= \left(2 - k^2(1.1)^{-2k}\right)\cos(20k)W_0,
\end{aligned}
\end{equation*}
where $W_0 = t_4 + 7t_5 + 6t_6 \in \h$. Note that $W$ is bounded.

Consider the dynamical system with state $X := (X_1,X_2) \in \g^2 =: \mc X$
\begin{equation*}
\begin{aligned}
	X_1^+ &= \frac{1}{2}\exp(W_1)X_1\exp(-W_1) - \exp(X_2)X_1\exp(-X_2)  + \frac{1}{2}\exp(W_2)X_2\exp(-W_2) \\
	X_2^+ &= \frac{1}{2}\exp(X_2)X_1\exp(-X_2) + \frac{1}{4}\exp(X_1 + W_1)X_2\exp(-(X_1 + W_1)),
\end{aligned}
\end{equation*}
where for all $Y \in \g$, $\exp(Y)X_i\exp(-Y) \in \g$~\cite[Propositions 2.16, 2.17]{Hall2015}. To see that these dynamics are indeed a Lie function, we use $\exp(Y)X_i\exp(-Y) = e^{\ad_Y}X_i$~\cite[Proposition 2.25]{Hall2015}:
\begin{equation*}
\begin{aligned}
	X_1^+ &= \left(\frac{1}{2}e^{\ad_{W_1}} - e^{\ad_{X_2}}\right)X_1  + \frac{1}{2}e^{\ad_{W_2}}X_2 \\
	X_2^+ &= \frac{1}{2}e^{\ad_{X_2}}X_1 + \frac{1}{4}e^{\ad_{X_1 + W_1}}X_2.
\end{aligned}
\end{equation*}

Recall $e^{\ad_Y} = \Id_\g + \ad_Y + \frac{1}{2!}\ad_Y^2 + \frac{1}{3!}\ad_Y^3 + \cdots$, yielding
\begin{equation*}
\begin{aligned}
	X_1^+ &= -\frac{1}{2}X_1 + \frac{1}{2}X_2 
		+  \sum_{\ell = 2}^\infty\frac{1}{\ell !}\left(\left(\frac{1}{2}\ad_{W_1}^\ell  - \ad_{X_2}^\ell \right)X_1 + \frac{1}{2}\ad_{W_2}^\ell X_2\right) \\
	X_2^+ &= \frac{1}{2}X_1 + \frac{1}{4}X_2 + \sum_{\ell = 2}^\infty\frac{1}{\ell !}\left(\frac{1}{2}\ad_{X_2}^\ell X_1 + \frac{1}{4}\ad_{X_1 + W_1}^\ell X_2\right).
\end{aligned}
\end{equation*}

Using the basis $\{t_1,t_2,t_3,t_4,t_5,t_6\}$ for $\g$, and letting $I_6 \in \Real^{6 \times 6}$ be the identity matrix, we can express the dynamics of $X$, as
\begin{equation*}
	X^+ = \underbrace{\left(\begin{bmatrix}-\frac{1}{2} & \frac{1}{2} \\ \frac{1}{2} & \frac{1}{4}\end{bmatrix} \otimes I_6\right)}_{\mathrm{Mat}A}X 
		+ \sum_{\ell = 2}^\infty\frac{1}{\ell !}\begin{bmatrix}\left(\frac{1}{2}\ad_{W_1}^\ell  - \ad_{X_2}^\ell \right)X_1 + \frac{1}{2}\ad_{W_2}^\ell X_2 \\
	\frac{1}{2}\ad_{X_2}^\ell X_1 + \frac{1}{4}\ad_{X_1 + W_1}^\ell X_2\end{bmatrix}.
\end{equation*}

We now verify that Assumption~\ref{asm:f} is satisfied. For all $Y \in \g$, $\|\ad_Y^\ell X_i\| \leq \mu^{\ell - 1}\|Y\|^{\ell - 1}\|X_i\|$, yielding
\begin{equation*}
	\|e^{\ad_Y}X_i\| \leq \sum_{\ell = 1}^\infty\frac{(\mu\|Y\|)^{\ell - 1}}{\ell !}\|X_i\| 
	= \frac{e^{\mu\|Y\|} - 1}{\mu\|Y\|}\|X_i\| < \infty,
\end{equation*}
so the dynamics of $X$ belong to class-$\mc A$, thereby satisfying Assumption~\ref{asm:f}.\ref{asm:ClassA}.

That $X = 0$ is an equilibrium is verified by substituting $X = 0$ into the dynamics. To verify that $X = 0$ is the only equilibrium, recall that the derived algebra is $\mathrm{Lie}_\Real\{t_4,t_5,t_6\}$, so a point is an equilibrium only if
\begin{equation*}
	P_0X = \underbrace{\left(\begin{bmatrix}-\frac{1}{2} & \frac{1}{2} \\ \frac{1}{2} & \frac{1}{4}\end{bmatrix} \otimes I_3\right)}_{\mathrm{Mat}\bar A_0}P_0X,
\end{equation*}
where $\rho(\bar A_0) = \bigsqcup_{i = 1}^3\left\{-\frac{3}{4},\frac{1}{2}\right\}$, implying that $\bar A_0$ is bijective. Therefore, a point can be an equilibrium only if $P_0X = 0$, or equivalently, $X \in \h^2$. As mentioned, $\h$ is isomorphic to the Heisenberg algebra, so the rest of the argument that Assumption~\ref{asm:f}.\ref{asm:vanish} is satisfied is similar to that in Example~\ref{ex:Nilpotent}.

It is clear from the form of $\mathrm{Mat}A$ that $A\mf h_i^2 \subseteq \mf h_i^2$. By Corollary~\ref{cor:Ainvariant}, this verifies Assumption~\ref{asm:f}.\ref{asm:invariance}.

From $\mathrm{Mat}A$, we find $\rho(A) = \bigsqcup_{i = 1}^6\left\{-\frac{3}{4},\frac{1}{2}\right\}$. Thus, by Theorem~\ref{thm:Solvable}, if the limiting supremum of $W$ is sufficiently small, then the origin of $\mc X$ is globally attractive. By Corollary~\ref{cor:SolvGAS}, if $W$ is bounded sufficiently small, then the origin is globally asymptotically stable. We illustrate simply that for the arbitrary choice of $W$ in this example, that $X \to 0$ as $k \to \infty$, as seen in Figure~\ref{fig:SolvableX}.
\begin{figure}[htb!]
\centering
\includegraphics[width=1\columnwidth]{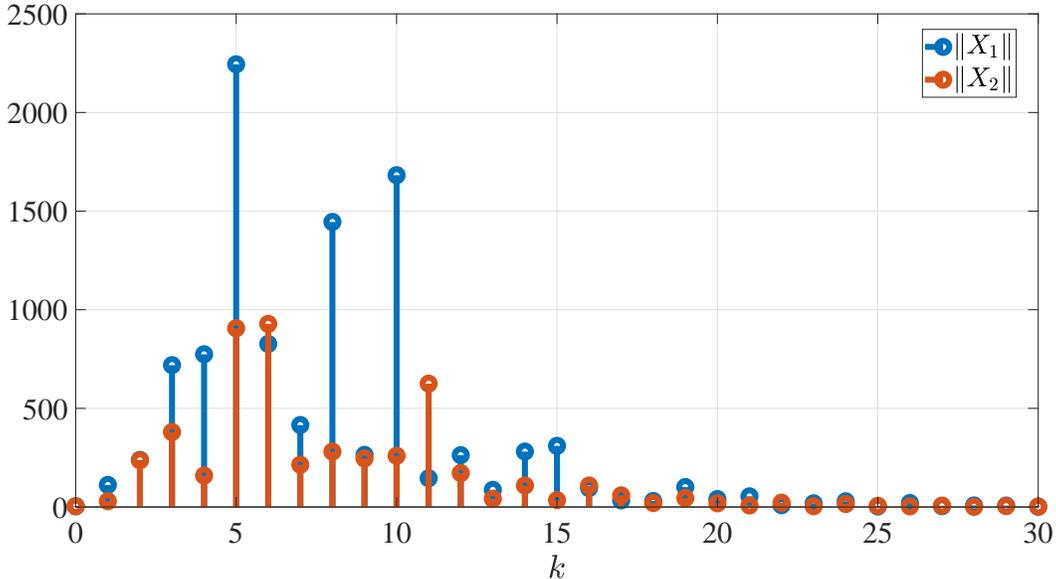}
\caption{The norms of the states $X_1,X_2 \in \g$.}
\label{fig:SolvableX}
\end{figure}
\exsymbol\end{example}

\section{Summary and Future Research}

We showed that for a class of systems evolving on solvable Lie algebras, global stability properties can be inferred from the linear part the dynamics. If the Lie algebra is solvable, then global asymptotic stability can be established. If the Lie algebra is nilpotent, then semiglobal exponential stability can be established. An interesting topic of future research would be strengthening the results in the more general, non-nilpotent case. 
Given an arbitrary finite-dimensional Lie algebra, it would be interesting to explore the use of the Levi decomposition to study the quotient dynamics on the radical, and see what utility this offers for studying stability on the full Lie algebra.

\section*{Appendix}

\subsection{Proof of \texorpdfstring{Claim~\ref{claim:u}}{Claim 1}}\label{app:u}
\begin{proof}[Claim~\ref{claim:u}]
Fix the word length $\ell \geq 2$ and the number of letters in $\widetilde X$, $1 \leq q \leq \ell$. There are $n^q$ choices of letters in $\widetilde X$, $r^{\ell - q}$ choices of letters in $\widetilde W$, and $\binom{\ell}{q}$ ways to position the letters in $\widetilde X$. Thus, there are $\binom{\ell}{q}n^qr^{\ell - q}$ words of length $\ell$ with $q$ letters in $\widetilde X$. First, recall from~\eqref{eq:ui}, that $u_i := \sum_{|\omega| \leq i}c_\omega \otimes (P_i\bar\omega_{i - 1})$. Applying~\eqref{eq:PwDecay}, we have
%

\begin{equation*}
	\|u_i[k]\| \leq \left(\sum_{\substack{2 \leq \ell \leq i \\ 1 \leq q \leq \ell}}\max_{|\omega| = \ell}\{\|c_\omega\|\}n^qr^{\ell - q}\mu^{\ell - 1}\|\imath_{i - 1}\|^\ell\alpha_{i - 1}^q\|\bar X_{i - 1}[0]\|^q\beta^{\ell - q}\right)\max_{\substack{2 \leq \ell \leq i \\ 1 \leq q \leq \ell}}\{\lambda_{i - 1}^qs^{\ell - q}\}^k,
\end{equation*}
whose right side equals
\begin{equation*}
	\overbrace{\left(\sum_{\ell = 2}^i\max_{|\omega| = \ell}\{\|c_\omega\|\}\mu^{\ell - 1}\|\imath_{i - 1}\|^\ell\sum_{q = 1}^\ell\binom{\ell}{q}n^qr^{\ell - q}\alpha_{i - 1}^qM^{q - 1}\beta^{\ell - q}\right)}^{=: \gamma_i} 
	\times{\underbrace{\max_{\substack{2 \leq \ell \leq i \\ 1 \leq q \leq \ell}}\{\lambda_{i - 1}^qs^{\ell - q}\}}_{\lambda_i}}^k\|\bar X_i[0]\|.
\end{equation*}

Since $0 < \lambda_{i - 1} < 1$ and $s \geq 1$, the maximization defining $\lambda_i$ is solved by $\ell = i$ and $q = 1$.
\end{proof}

\subsection{Proof of \texorpdfstring{Lemma~\ref{lem:SolvBracket}}{Lemma 5.2}}\label{app:SolvBracket}
\begin{proof}[Lemma~\ref{lem:SolvBracket}]
Using $\Id_\g - \imath_{i - 1} \circ P_{i - 1} + \imath_{i - 1} \circ P_{i - 1} = \Id_\g$ and  bilinearity of the Lie bracket,
\begin{equation}\label{eq:SolvPw}
	P_i\omega 
	= P_i[(\Id_\g - \imath_{i - 1} \circ P_{i - 1})Y_1,[Y_2,[\ldots,Y_{|\omega|}]\cdots] 
		+ P_i[\imath_{i - 1} \circ P_{i - 1}Y_1,[Y_2,[\ldots,Y_{|\omega|}]\cdots],
		\qquad Y \in \widetilde X \cup \widetilde W.
\end{equation}

We next decompose the second letter of the first term in~\eqref{eq:SolvPw} with respect to $\imath_0 \circ P_0$ and invoke Lemma~\ref{lem:CompProj}:
\begin{multline}\label{eq:SolvWordDecomp}
	P_i[(\Id_\g - \imath_{i - 1} \circ P_{i - 1})Y_1,[Y_2,[\ldots,Y_{|\omega|}]\cdots]
	= P_i[(\Id_\g - \imath_{i - 1} \circ P_{i - 1})Y_1,[\imath_0 \circ P_0 Y_2,[\ldots,Y_{|\omega|}]\cdots] \\
	+ P_i\underbrace{[\underbrace{(\Id_\g - \imath_{i - 1} \circ P_{i - 1})Y_1}_{\in \h^{(i)}},[\underbrace{(\Id_\g - \imath_0 \circ P_0)Y_2}_{\in \h^{(1)}},[\ldots,Y_{|\omega|}]\cdots]}_{\in \h^{(i + 1)}},
\end{multline}
where membership in $\h^{(i + 1)}$ follows from Theorem~\ref{thm:StronglyCentral}; the second term is zero, since $P_i\h^{(i + 1)} = 0$. Decomposing the rest of the letters in~\eqref{eq:SolvWordDecomp} with respect to $\imath_0 \circ P_0$ yields
\begin{equation}\label{eq:LastTwoLetters}
	P_i[(\Id_\g - \imath_{i - 1} \circ P_{i - 1})Y_1,[Y_2,[\ldots,Y_{|\omega|}]\cdots] =
	P_i[(\Id_\g - \imath_{i - 1} \circ P_{i - 1})Y_1,[\imath_0 \circ P_0 Y_2,[\ldots,\imath_0 \circ P_0 Y_{|\omega|}]\cdots].
\end{equation}

Now decompose the second letter of the second term in~\eqref{eq:SolvPw} with respect to $\imath_{i - 1} \circ P_{i - 1}$:
\begin{multline}
	P_i[\imath_{i - 1} \circ P_{i - 1}Y_1,[Y_2,[\ldots,Y_{|\omega|}]\cdots] 
		= P_i[\imath_{i - 1} \circ P_{i - 1}Y_1,[\overbrace{(\Id_\g - \imath_{i - 1} \circ P_{i - 1})Y_2}^{\in \h^{(i)}},[Y_3,[\ldots,Y_{|\omega|}]\cdots] \\
			+ P_i[\imath_{i - 1} \circ P_{i - 1}Y_1,[\imath_{i - 1} \circ P_{i - 1} Y_2,[Y_3,[\ldots,Y_{|\omega|}]\cdots].
\end{multline}

We continue in a fashion similar to that following~\eqref{eq:SolvPw}, the only noteworthy difference is the decomposition of $\imath_{i - 1} \circ P_{i - 1}Y_1$ with respect to $\imath_0 \circ P_0$.

\begin{claim}\label{claim:PiP}
For all $i \geq 1$, the following diagram commutes.
\begin{equation*}
	\xymatrix{\mc \g \ar[drr]_{P_0} \ar[r]^{P_{i - 1}} 	&	\g / \h^{(i)} \ar[r]^{\imath_{i - 1}} & \g \ar[d]^{P_0} \\
			& &	\g / \h}
\end{equation*}
\end{claim}

\begin{proof}[Claim~\ref{claim:PiP}]
From the definitions of $P_0$, $P_{i - 1}$, and $\imath_{i - 1}$, we have $\g = \image \imath_{i - 1} \oplus \h^{(i)}$ and $\Ker P_0 = \h \supseteq \h^{(i)} = \Ker P_{i - 1}$. Then $P_0\g = P_0 \image \imath_{i - 1} \oplus P_0 \h^{(i)} = P_0 \image \imath_{i - 1}$.
\end{proof}

It follows immediately from Claim~\ref{claim:PiP} that $\imath_0 \circ P_0 \circ \imath_{i - 1} \circ P_{i - 1} = \imath_0 \circ P_0$. Thus, the decomposition process specified above yields
\begin{multline}\label{eq:Solv12}
	P_i\omega 
		= P_i[\imath_{i - 1} \circ P_{i - 1} Y_1,[\imath_{i - 1} \circ P_{i - 1} Y_2,[Y_3,[\ldots,Y_{|\omega|}]\cdots] \\
			+ P_i[(\Id_\g - \imath_{i - 1} \circ P_{i - 1})Y_1,[\imath_0 \circ P_0 Y_2,[\ldots,\imath_0 \circ P_0 Y_{|\omega|}]\cdots] \\
				+ P_i[\imath_0 \circ P_0Y_1,[(\Id_\g - \imath_{i - 1} \circ P_{i - 1})Y_2,[\imath_0 \circ P_0 Y_3,[\ldots,\imath_0 \circ P_0 Y_{|\omega|}]\cdots].
\end{multline}

Applying this process to the rest of the letters in the first word of~\eqref{eq:Solv12} completes the proof.
\end{proof}

\subsection{Proof of \texorpdfstring{Claim~\ref{claim:Converges}}{Claim 2}}\label{app:Converges}
\begin{proof}[Claim~\ref{claim:Converges}]
Suppose $f$ satisfies~\eqref{eq:MyConvergence}. In particular, suppose there exists $\varrho_1 \leq 1$ such that $$\|X_1\|,\ldots,\|X_n\|,\|W_1\|,\ldots,\|W_r\| < \varrho_1.$$ On this domain, we have 
%
	$\|\omega\| \leq \mu^{|\omega| - 1}\varrho_1^{|\omega|}$
%
and
\begin{equation*}
	\sum_\omega\mu^{|\omega| - 1}\|c_\omega\|\varrho_1^{|\omega|} < \infty.
\end{equation*}
We can rewrite this summation by grouping all words of the same length:
\begin{equation*}
	\sum_{\ell = 2}^\infty\mu^{\ell - 1}\left(\sum_{|\omega| = \ell}\|c_\omega\|\right)\varrho_1^\ell,
\end{equation*}
which can be viewed as a series over the single index $\ell$. Since this series converges, by the root test~\cite[Theorem 3.33]{Rudin1976}, we have
\begin{equation*}
\begin{aligned}
	\limsup_{\ell \to \infty}\sqrt[\ell]{\mu^{\ell - 1}\varrho_1^\ell\sum_{|\omega| = \ell}\|c_\omega\|}
	&= \varrho_1\limsup_{\ell \to \infty}\mu^{1 - \frac{1}{\ell}}\limsup_{\ell \to \infty}\sqrt[\ell]{\sum_{|\omega| = \ell}\|c_\omega\|} \\
	&= \varrho_1\mu\limsup_{\ell \to \infty}\sqrt[\ell]{\sum_{|\omega| = \ell}\|c_\omega\|} \\
	&\leq 1.
\end{aligned}
\end{equation*}

Let $0 < \varrho_2 < \frac{\varrho_1}{\|\imath_0\|}$. Applying the root test to the series
\begin{equation}\label{eq:AConvergenceomega}
	\sum_\omega(\mu\|\imath_0\|)^{|\omega| - 1}|\omega|\|c_\omega\|\varrho_2^{|\omega|},
\end{equation}
we have 
\begin{equation*}
\begin{aligned}
	\limsup_{\ell \to \infty}\sqrt[\ell]{\ell(\mu\|\imath_0\|)^{\ell - 1}\varrho_2^\ell\sum_{|\omega| = \ell}\|c_\omega\|}
	&= \varrho_2\mu\|\imath_0\|\limsup_{\ell \to \infty}\sqrt[\ell]{\ell}\limsup_{\ell \to \infty}\sqrt[\ell]{\sum_{|\omega| = \ell}\|c_\omega\|} \\
	&= \varrho_2\mu\|\imath_0\|\limsup_{\ell \to \infty}\sqrt[\ell]{\sum_{|\omega| = \ell}\|c_\omega\|} \\
	&< \varrho_1\mu\limsup_{\ell \to \infty}\sqrt[\ell]{\sum_{|\omega| = \ell}\|c_\omega\|} \\
	&< 1,
\end{aligned}
\end{equation*}
which implies that~\eqref{eq:AConvergenceomega} converges. Let $\varrho \leq \varrho_2^2$, then for all $|\omega| \geq 2$, $\varrho^{|\omega| - 1} < \varrho_2^{|\omega|}$. Then, by the comparison test~\cite[Theorem 3.25]{Rudin1976}, if $\|\bar X_0\|,\|\bar W_0\| \leq \varrho$, then~\eqref{eq:OmegaSeries} converges.
\end{proof}

\bibliographystyle{spphys}
\bibliography{references}
\end{document}